\documentclass[12pt]{article}

\usepackage[T1]{fontenc}
\usepackage[utf8]{inputenc}
\usepackage[english]{babel}
\usepackage{amsmath}
\usepackage{amsthm}
\usepackage{amssymb}
\usepackage{tikz}
\usepackage{graphicx}
\usepackage{faktor}
\usepackage{xcolor, colortbl}
\usepackage{makecell}
\usepackage{caption}
\usepackage{subcaption}
\usepackage{comment}
\usepackage{setspace}
\usepackage[12pt]{moresize}
\usepackage{wasysym}
\usepackage{float}
\usepackage[margin=0.5in]{geometry}
\usepackage{mathrsfs}
\usepackage{url}
\usepackage{authblk}

\DeclareMathOperator{\dom}{dom}
\DeclareMathOperator{\id}{id}
\DeclareMathOperator{\dist}{dist}
\DeclareMathOperator{\opt}{opt}
\DeclareMathOperator{\conv}{conv}
\DeclareMathOperator{\Lip}{Lip}
\DeclareMathOperator{\Tr}{Tr}

\newcommand{\fr}{\penalty-20\null\hfill$\blacksquare$}

\theoremstyle{plain}
\newtheorem{thm}{Theorem}[section]

\newtheorem*{thm*}{Theorem}
\newtheorem{lem}[thm]{Lemma}
\newtheorem{prop}[thm]{Proposition}
\newtheorem{cor}[thm]{Corollary}

\theoremstyle{remark}
\newtheorem{oss}[thm]{Remark}
\newtheorem{ex}[thm]{Example}

\numberwithin{equation}{section}

\begin{document}
	
	\title{\LARGE{\textbf{Partial regularity for minimizers of a class \\of discontinuous Lagrangians}}}
	\author{\textbf{Roberto Colombo}}
	\affil{Scuola Normale Superiore \\Piazza dei Cavalieri, 7, 56126 Pisa, Italy\\ e-mail: roberto.colombo@sns.it}
	\date{}
	\maketitle

	\abstract
	{\noindent We study a one dimensional Lagrangian problem including the variational reformulation, derived in a recent work of Ambrosio-Baradat-Brenier, of the discrete Monge-Ampère gravitational model, which describes the motion of interacting particles whose dynamics is ruled by the optimal transport problem. The more general action type functional we consider contains a discontinuous potential term related to the descending slope of the opposite squared distance function from a generic discrete set in $\mathbb{R}^{d}$. We exploit the underlying geometrical structure provided by the associated Voronoi decomposition of the space to obtain $C^{1,1}$ regularity for local minimizers out of a finite number of shock times.}
	
	\tableofcontents
	
	\section{Introduction}  
	In recent years, action functionals of the form
	\begin{equation}\label{primafunc}
		I_{f}(\gamma)=\int_{0}^{1}|\dot{\gamma}|^{2}+|\nabla f(\gamma)|^2
	\end{equation}
	have received the attention of many authors, due to their appearence in several areas of Mathematics. In the theory of gradient flows, for instance, they correspond to the integral form of the energy dissipation (see \cite{AGS}), and they are also related to the so called entropic regularization of the Wasserstein distance, when $f$ is a multiple of the logarithmic entropy, defined on the space of probability measures with finite quadratic moment (see \cite{Clerc2020OnTV}). In all these cases, the main obstruction to the application of standard results of Calculus of Variations stems from the lack of differentiability, even continuity, of the Lagrangian with respect to $\gamma$.
	
	When $f:X\rightarrow (-\infty,+\infty]$ is a $\lambda$-convex function defined on a metric space $(X,d)$, the term $|\nabla f(x)|$ has to be interpreted as the descending slope of $f$ at $x$, namely
	\begin{equation}\label{descendingslope}
		|\nabla^{-} f|(x):=\underset{y\rightarrow x}{\limsup}\,\frac{[f(x)-f(y)]^{+}}{d(x,y)},
	\end{equation}
	and, if $X$ is a Hilbert space, it also coincides with the norm of the minimal selection in the subdifferential $\partial f(x)$, known as the extended gradient of $f$ at $x$ (see \cite{AGS}). In this general framework, stability of the functionals $I_{f}$ with respect to $\Gamma$-convergence of the functions $f$ was investigated in \cite{AmbBrMAG}, \cite{AmbBrGAMMACONV} and \cite{AmbBre21}. In particular, \cite{AmbBrMAG} addressed a rigorous derivation, along the lines of \cite{Brenier2015ADL}, of a dynamical system of interacting particles strictly related to the optimal transport problem, known as the discrete Monge-Ampère gravitational (MAG) model (see subsection \ref{subsectionMAG}). What emerged from this work is that the dynamics of MAG can be conveniently studied as the Euler-Lagrange equation associated to an action functional of type (\ref{primafunc}), where $f=f_{K}$ is the $(-1)$-convex function given by the opposite squared distance from a specific discrete set $K\subset \mathbb{R}^{d}$, namely
	\begin{equation}\label{oppdistintro}
		f_{K}(x):=-\frac{\dist^{2}_{K}(x)}{2}=-\underset{y\in K}{\min}\,\frac{|x-y|^{2}}{2}.
	\end{equation}
	Clearly, $f_{K}$ is not everywhere differentiable in general, so that $\nabla f_{K}$, denoting the extended gradient of $f_K$, is not even continuous, and standard results of Calculus of Variations are not directly applicable in this case.
	
	The present work aims at a systematic analysis of the properties of local minimizers for the functional $I_{f_{K}}$, where $K$ is a generic discrete set $K$ in $\mathbb{R}^{d}$. Our results apply in particular to solutions of the discrete MAG model, thereby addressing the general $n$-dimensional case, left open in \cite{AmbBrMAG}, where the most involved part of the analysis was carried out only in dimension $1$. 
	
	The plan of the paper is the following. In section \ref{secgeneralframe} we present the general framework and the motivations of our work. More precisely, in subsection \ref{subsecactionfunc} we provide a contextualization of the problem in the general Hilbertian setting, with particular emphasis on the variational properties of functionals of type (\ref{primafunc}), when $f$ is a $\lambda$-convex function. The main references for this part are \cite{AmbBrGAMMACONV} and \cite{AGS}. Then, in subsection \ref{subsectionMAG}, we introduce the Monge-Ampère gravitational model in the flat torus $\mathbb{T}^{n}$ as a modification of the classical Newtonian gravitation in which the linear Poisson equation is replaced by the fully nonlinear Monge-Ampère equation (see (\ref{MAG})). Following the ideas of \cite{Brenier2015ADL}, we underline the intriguing link of this dynamical system with the optimal transport problem, whose powerful tools can be used in order to derive a Lagrangian reformulation of MAG particularly meaningful in the discrete setting, where additional foundation to the model is given by the results in \cite{AmbBrMAG}. By means of a least action principle, we then interpret solutions of the discrete MAG model as local minimizers of the functional $I_{f_{K}}$, where $f_{K}$ is the opposite squared distance function from a discrete set $K\subset \mathbb{R}^{d}$, as defined in (\ref{oppdistintro}). 
	Section \ref{secoppsqdistfunc}, being the core of the paper, is then devoted to the analysis of local minimizers for the Lagrangian problem associated to $I_{f_{K}}$, when $K$ is a finite collection of points in $\mathbb{R}^{d}$:
	\[K=\left\{p_{1},\dots,p_{N}\right\}.\]
	We crucially consider the Voronoi partition of the space carried by $K$, which encodes the underlying geometrical structure of the problem, and exploit it in order to obtain, in Proposition \ref{conserv}, the existence of some specific directions along which momentum is locally conserved by the dynamics. As a byproduct, we show in Corollary \ref{reginvoronoicell} that a local minimizer $\gamma$ is regular as long as it stays in a single Voronoi cell, possibly developing singularities only at those times in which the optimality class changes. As it is shown later, the set $K$ also carries a partition of $\mathbb{R}^{d}$ into \lq\lq potential zones'' (see Proposition \ref{vorcellsandpotzones}). This second partition is in general less fine than the Voronoi one, and coincides with it when the set $K$ is \lq\lq balanced'', like for instance a cubic lattice. We then define $S(\gamma)$ to be the set of \lq\lq shock times'', at which the curve $\gamma$ jumps from a Voronoi cell to another, and $NDS(\gamma)\subseteq S(\gamma)$ the set of \lq\lq nondegenerate shock times'', at which $\gamma$ not only changes the Voronoi cell, but also the potential zone. With this in mind, our main regularity results Theorem \ref{teoprinc} and Corollary \ref{improvedreg} can be collected in a single statement as follows:
	\begin{thm*}[Partial regularity]
		Let $\gamma$ be a local minimizer of $I_{f_K}$ with endpoints constraints. Then
		\begin{itemize}
			\item[i)] $\gamma$ has a finite number of nondegenerate shock times out of which it is $C^{1,1}$.
			\item[ii)] Under the additional assumption that $K$ is balanced, $\gamma$ has a finite number of shock times out of which it is $C^{\infty}$.
		\end{itemize}
	\end{thm*}
	\noindent This result in particular provides an extention to any space dimension of \cite[Theorem 13]{AmbBrMAG}, where regularity out of a finite number of shock times was proved for minimizers of a one dimentional version of the MAG model.    
	
	\bigskip
	
	\noindent\textbf{Acknowledgements.} I wish to thank Prof.\ Luigi Ambrosio for introducing me to the present research topic and for providing me with very useful suggestions throughout the preparation of this work.
	
	\section{General framework and motivations}\label{secgeneralframe}
	
	\subsection{Action functionals depending on the gradient of convex functions}\label{subsecactionfunc}
	\paragraph{$\lambda$-convex functions.} 
	Given a Hilbert space $H$, we consider a function $f:H\rightarrow (-\infty,+\infty]$, and denote by $\dom(f)$ its finiteness domain. We say that $f$ is \textit{$\lambda$-convex} if $x\mapsto f(x)-\frac{\lambda}{2}|x|^2$ is convex. It is easily seen that $\lambda$-convex functions are precisely those functions that satisfy the perturbed convexity inequality
	\[f((1-t)x+ty)\le (1-t)f(x)+tf(y)-\frac{\lambda}{2}t(1-t)|x-y|^{2},\quad t\in[0,1].\]
	
	\noindent By $\partial f(x)$ we denote the \textit{Gateaux subdifferential} of $f$ at $x\in \dom(f)$, namely the (possibly empty) closed convex set
	\[\partial f(x):=\left\{\xi \in H: \underset{t\rightarrow 0^{+}}{\liminf}\,\frac{f(x+tv)-f(x)}{t}\ge \xi\cdot v, \, \forall v \in H\right\}.\]
	We denote by $\dom(\partial f)$ the domain of the subdifferential. For a $\lambda$-convex function, we can exploit the monotonicity of difference quotients to derive the equivalent non asymptotic definition of the subdifferential
	\[\partial f(x):=\left\{\xi\in H: f(y)\ge f(x)+\langle \xi,y-x \rangle + \frac{\lambda}{2}|y-x|^{2}, \, \forall y \in H\right\}.\]
	Whenever $x\in \dom(\partial f)$, there exists a unique element $\xi$ with minimal norm in $\partial f(x)$, obtained by projecting $0$ on $\partial f(x)$. This element is called the \textit{extended gradient} of $f$ at $x$, and is denoted by $\nabla f(x)$. The concept of extended gradient is strictly related to the one of \textit{descending slope} of $f$ at $x\in \dom(f)$, namely
	\[|\nabla^{-}f|(x):=\underset{y\rightarrow x}{\limsup}\,\frac{[f(x)-f(y)]^{+}}{|x-y|}.\]
	In fact, for $\lambda$-convex functions, it can be proved that $\partial f(x)$ is not empty if and only if $|\nabla^{-}f|(x)<+\infty$, and that, in this case, the following equalities hold (see \cite{AGS}):
	\begin{equation}\label{extgrad}
		|\nabla f(x)|=|\nabla^{-}f|(x)=\underset{y\neq x}{\sup}\,\frac{[f(x)-f(y)+\frac{\lambda}{2}|x-y|^{2}]^{+}}{|x-y|}.
	\end{equation}
	By setting $|\nabla f|$ equal to $+\infty$ out of $\dom(\partial f)$, we easily deduce from (\ref{extgrad}) that the function $H\ni x \mapsto |\nabla f(x)|\in [0,+\infty]$ is lower semicontinuous, being the supremum of a collection of continuous functions. 
	
	In this paper we deal with the following specialization of the above setting. Given a closed set $K\subseteq H$, we consider the \textit{opposite squared distance function} from $K$, namely
	\begin{equation}\label{oppdistfunc}
		f_{K}(x):=-\frac{\dist^{2}_{K}(x)}{2}=-\underset{y\in K}{\inf}\frac{|x-y|^{2}}{2}.
	\end{equation}
	The infimum is not attained in general, unless $K$ is either convex or compact, or $H$ is finite dimensional.
	By defining the convex function 
	\[g_{K}(x):=\underset{y\in K}{\sup}\left(\langle x, y\rangle-\frac{|y|^2}{2}\right),\]
	we derive from the equality $g_{K}(x)=f_{K}(x)+\frac{|x|^{2}}{2}$ that $f_{K}$ is $(-1)$-convex.
	\paragraph{A class of action functionals.}
	We now introduce, in the general Hilbertian setting, the class of action functionals that we are going to study throughout the paper. We fix a function $h:[0,+\infty]\rightarrow [0,+\infty]
	$ representing a \lq\lq potential shape''. Then, for $\delta>0$ and $f:H\rightarrow (-\infty,+\infty]$ proper, $\lambda$-convex and lower semicontinuous, we consider the functional $I^{\delta}_{f}:C([0,\delta], H)\rightarrow [0,+\infty]$ defined by
	\begin{equation}\label{funzionale}
		I^{\delta}_{f}(\gamma):=
		\begin{cases}
			\displaystyle\int_{0}^{\delta}|\dot{\gamma}|^2+h(|\nabla f|^{2}(\gamma))\quad &\text{if $\gamma \in AC([0,\delta], H)$},\\
			+\infty  &\text{otherwise}.	
		\end{cases}
	\end{equation}
	Compared to the one of type (\ref{primafunc}) studied so far in the literature, we consider here the enriched class of functionals in which the potential shape $h$ is allowed to be different from the identity. In the sequel we assume $h$ to be continuous, and $C^{1}$ when restricted to $[0,+\infty)$. 
	
	As we have in mind to study this type of functionals from the variational point of view, it is crucial to realize that (\ref{funzionale}) is lower semicontinuous with respect to the $C([0,\delta],H)$ topology. This in fact easily follows from the lower semicontinuity of the classical action and the above characterization of the extended gradient (\ref{extgrad}). Then, for $x_{0}, x_{\delta} \in H$, the infimum
	\begin{equation*}\label{infimum}
		\Gamma^{\delta}_{f}(x_0,x_{\delta}):=\inf \left\{I^{\delta}_{f}(\gamma):\gamma(0)=x_{0}, \gamma(\delta)=x_{\delta}\right\}
	\end{equation*}
	is attained under suitable coercivity conditions. Note in particular that this is the case if $H$ is finite dimensional. 
	
	Due to the lack of continuity of the potential term, however, very little is known about the regularity for minimizers of this type of functionals, even in the finite dimensional case. We could ask for instance whether some higher regularity or at least a sort of \textit{Euler-Lagrange equation} like formally
	\begin{equation}\label{eulerolagrange}
		\ddot{\gamma}=h'(|\nabla f(\gamma)|^2)\nabla^{2}f(\gamma)\nabla f(\gamma)
	\end{equation}
	could be derived for local minimizers. In the very specific case in which $f$ is the opposite squared distance function from a discrete set in $\mathbb{R}^{d}$, we will prove in the sequel that local minimizers are piecewise $C^{1,1}$, and that, out of a finite number of singularities, (\ref{eulerolagrange}) holds taking the modulus on both sides and replacing the equality with a $\le$ sign (see Theorem \ref{teoprinc}). Nevertheless, in the general setting, one can exploit the fact that the functional $I^{\delta}_{f}$ is autonomous in order to perform \lq\lq horizontal'' variations of the independent variable, and eventually derive the \textit{Dubois-Reymond equation} for a local minimizer $\gamma$ (see \cite{AMBROSIO1989301}):
	\begin{equation}\label{DBR}
		\frac{d}{dt}\left\{|\dot{\gamma}|^2-h(|\nabla f|^{2}(\gamma))\right\}=0
	\end{equation}
	in the sense of distributions in $(0,\delta)$. Equivalently, there exists a constant $c\in \mathbb{R}$ such that 
	\[|\dot{\gamma}|^2=h(|\nabla f|^{2}(\gamma))+c\]
	a.e.\ in $(0,\delta)$. This implies in particular that every local minimizer of $I^{\delta}_{f}$ is Lipschitz continuous, provided that $|\nabla f|$ is bounded on bounded sets.
	
	We end this part by quoting a result from \cite{AmbBrGAMMACONV} addressing the matter of stability for the class of functionals considered so far. By adding endpoints constraints $x_{0},x_{\delta}\in H$, we define the functional $I^{\delta}_{f,x_{0},x_{\delta}}:C([0,\delta], H)\rightarrow [0,+\infty]$ such that
	\begin{equation*}\label{funzionaleconendpoits}
		I^{\delta}_{f,x_{0},x_{\delta}}(\gamma):=
		\begin{cases}
			\displaystyle\int_{0}^{\delta}|\dot{\gamma}|^2+h(|\nabla f|^{2}(\gamma))\quad &\text{if $\gamma \in AC([0,\delta], H)$ and $\gamma(0)=x_{0}, \gamma(\delta)=x_{\delta}$},\\
			+\infty  &\text{otherwise}.  	
		\end{cases}
	\end{equation*} 
	\begin{thm}[Stability, \cite{AmbBrGAMMACONV}]
		\label{stab}
		Let $f_{j}, f$ be uniformly $\lambda$-convex functions, and let $x_{j,0}, x_{j,\delta}, x_{0}, x_{\delta}\in H$. Suppose that
		\begin{itemize}
			\item[i)]$f_{j}\rightarrow f$ w.r.t. Mosco convergence.
			\item[ii)]$\underset{j\rightarrow \infty}{\lim}x_{j,i}=x_{i}$, for $i=0, \delta$.
			\item[iii)]$\underset{j}{\sup}|\nabla f_{j}|(x_{j,i})<\infty$, for $i= 0, \delta$.
		\end{itemize}
		Then $I^{\delta}_{f_{j},x_{j,0},x_{j,\delta}}$ $\Gamma$-converge to $I^{\delta}_{f,x_{0},x_{\delta}}$ in the $C([0,\delta], H)$ topology. 
	\end{thm}
	\noindent As a byproduct, under an additional equi-coercivity assumption, this theorem grants convergence of minimal values to minimal values and of minimizers to minimizers. Notice that Theorem \ref{stab} is stated in \cite{AmbBrGAMMACONV} for $h=\id$, but the same proof is seen to work in the general case with only minor modifications. 
	\subsection{The Monge-Ampère gravitational model}\label{subsectionMAG}
	In a periodic spatial domain like the flat torus $\mathbb{T}^{n}=\mathbb{R}^{n}/\mathbb{Z}^{n}$, we can describe classical Newtonian gravitation of a unity of mass in a \lq\lq parametric'' way as follows. We first choose a reference metric probability space $(\mathcal{A}, \lambda)$ of labels for the gravitating particles. Then we assign to each particle $a\in \mathcal{A}$ its position $X_{t}(a)\in \mathbb{T}^{n}$ at time $t$. Tipical choices for the reference space are the unit cube $[0,1]^n$ with the $n$-dimensional Lebesgue measure in the continuous case, and a finite set of points with the renormalized counting measure in the discrete case. Denoting by $\mu_{t}:=(X_{t})_{\#}\lambda$ the image measure of $\lambda$ by $X_t$, the Newtonian model can be written as
	\begin{equation}\label{Newton}
		\begin{cases}
			\frac{d^{2}}{dt^{2}}X_{t}(a)=-\nabla \phi_{t}(X_{t}(a)),\\
			\Delta \phi_{t} = \mu_{t}-1.
		\end{cases}
	\end{equation}  
	Here $\phi_{t}$ is the gravitational potential generated by $\mu_{t}$, defined on $\mathbb{T}^{n}$.
	Note that due to the periodicity of the space, the average density $1$ has been removed from the right hand side of the Poisson equation, in order to let the uniform measure $\mathscr{L}^{n}$ be a stationary solution of the system. This is a perfectly meaningful assumption, because by symmetry, the attractive force of the uniform density has to be zero everywhere on $\mathbb{T}^{n}$. 
	
	In this section we are interested in the related \textit{Monge-Ampère gravitational model} (MAG in short), which is simply obtained from (\ref{Newton}) by replacing the Poisson equation with the fully nonlinear Monge-Ampère equation:
	\begin{equation}\label{MAG}
		\begin{cases}
			\frac{d^{2}}{dt^{2}}X_{t}(a)=-\nabla \phi_{t}(X_{t}(a)),\\
			\det(\mathbb{I}+\nabla^{2}\phi_{t}) = \mu_{t}.
		\end{cases}
	\end{equation} 
	Notice that (\ref{Newton}) can be recovered from (\ref{MAG}) by expanding the determinant in the Monge-Ampère equation and keeping only the linear term:
	\[\det(\mathbb{I}+\nabla^{2}\phi_{t})\approx 1 + \Tr(\nabla^{2}\phi_t)=1+\Delta \phi_{t}.\]
	We refer to \cite{Brenier2015ADL} and the references therein for a broader introduction to this dynamical system, as well as for a comparison with the classical Newtonian model.
	
	\paragraph{The MAG model in optimal transportation terms.}
	System (\ref{MAG}) appears to have an intriguing geometrical interpretation if we look at it from the optimal transportation point of view. In order to better illustrate this link, we first quote the following specialization to the flat torus $\mathbb{T}^{n}$ of the classical Brenier-McCann theorem on the existence and uniqueness of optimal transport maps on Riemannian manifolds (see \cite{Brenier1991PolarFA}, \cite{CorderoErausquin1999SurLT}, \cite{McCann2001}). Let us begin with some notation. We denote by $\pi:\mathbb{R}^{n}\rightarrow \mathbb{T}^{n}$ the projection to the quotient. We say that a vector field $F:\mathbb{R}^{n}\rightarrow \mathbb{R}^{n}$ is $\mathbb{Z}^{n}$-translation invariant if $F(\cdot+z)=F(\cdot)+z$, for every $z\in \mathbb{Z}^{n}$. If this is the case we make a little abuse of notation by considering $F$ also as a vector field from $\mathbb{T}^{n}$ to itself. Given a Borel probability measure $\lambda$ on $\mathbb{T}^{n}$, we consider the Hilbert space 
	\[H_{\lambda}:= L^{2}(\mathbb{T}^{n}, \lambda; \mathbb{R}^{n})\]
	and its closed subset $K_{\lambda}$ given by all the $\lambda$-preserving vector fields 
	\[K_{\lambda}:=\left\{Y\in H_{\lambda}: (\pi \circ Y)_{\#}\lambda = \lambda \right\}.\]
	Finally, we recall that $f_{K_{\lambda}}$ denotes the opposite squared distance function from $K_{\lambda}$, as defined in (\ref{oppdistfunc}). 
	\begin{thm}[Existence and uniqueness of optimal transport maps in $\mathbb{T}^{n}$]
		\label{CorderoErasquin}
		Let $\mu$ and $\lambda$ be Borel probability measures on $\mathbb{T}^{n}$, and suppose that $\mu \ll \mathscr{L}^{n}$. Then 
		\begin{itemize}
			\item[i)] There exists a locally Lipschitz convex function $\psi:\mathbb{R}^{n}\rightarrow \mathbb{R}$, such that $\phi(x):=\psi(x)-\frac{|x|^{2}}{2}$ is $\mathbb{Z}^{n}$-periodic (therefore $\nabla \psi(x)=x + \nabla \phi(x)$ is $\mathbb{Z}^{n}$-translation invariant), and $T:=\nabla \psi : \mathbb{T}^{n}\rightarrow \mathbb{T}^{n}$ is the unique optimal transport map from $\mu$ to $\lambda$.
			\item[ii)] If $\mu = \rho \mathscr{L}^{n}$ and $\lambda= \eta \mathscr{L}^{n}$ are both absolutely continuous w.r.t.\ the Lebesgue measure, then $\phi$ solves the Monge-Ampère equation  
			\begin{equation}\label{mongeampere}
				\det(\mathbb{I}+\nabla^{2}\phi)\eta(T(x))=\rho(x)
			\end{equation}
			in the almost everywhere sense. Furthermore, if $\rho$ and $\eta$ are of class $C^{0,\alpha}$, then $\phi$ is of class $C^{2,\beta}$, for $0<\beta<\alpha$, and solves (\ref{mongeampere}) in the classical sense.
			\item[iii)] Let $Y\in H_{\lambda}$ be such that $(\pi \circ Y)_{\#}\lambda = \mu$. Then $T \circ Y$ is the unique projection of $Y$ on $K_{\lambda}$, and
			\begin{equation}\label{wassersteindist}
				\lVert T \circ Y - Y \rVert_{H_{\lambda}} = W_{2}(\mu,\lambda),
			\end{equation}
			where $W_{2}$ is the Wasserstein distance in the space $\mathscr{P}_{2}(\mathbb{T}^{n})$. Moreover, the map $f_{K_{\lambda}}$ is Gateaux differentiable at $Y$ and it holds
			\begin{equation}\label{nablaoppdist}
				\nabla f_{K_{\lambda}}(Y)= T \circ Y - Y= \nabla \phi \circ Y.
			\end{equation}
		\end{itemize}
	\end{thm}
	\noindent In order to reformulate the Monge-Ampère gravitational model in optimal transportation terms, we look to the continuous case, in which the reference space is given by $(\mathbb{T}^{n},\mathscr{L}^{n})$. Fix then $\lambda = \mathscr{L}^{n}$ in the Theorem above, and consider a parametrization $X_{t}:\mathbb{T}^{n}\rightarrow \mathbb{T}^{n}$ such that $(X_{t})_{\#}\lambda = \mu_{t}$ and $\mu_{t}=\rho_{t}\mathscr{L}^{n}$ is absolutely continuous w.r.t.\ the Lebesgue measure. If $Y_{t}\in H_{\lambda}$ is any lifting of $X_{t}$, that is to say a map that satisfies $\pi \circ Y_{t} = X_{t}$, then Theorem \ref{CorderoErasquin} grants that the Kantorovich potential $\phi_{t}$ solves the Monge-Ampère equation
	\[\det(\mathbb{I}+\nabla^{2}\phi_{t})=\mu_{t},\]
	and $-\nabla \phi_{t} = Y_{t}-T_{t}\circ Y_{t}$, where $T_{t}$ is the unique optimal transport map from $\mu_{t}$ to $\lambda$. So we see that (\ref{MAG}) reduces to 
	\begin{equation}\label{MAG2}
		\frac{d^{2}}{dt^{2}}Y_{t}(a)= Y_{t}(a)-T_{t}(Y_{t}(a)),
	\end{equation}
	with $T_{t}$ equal to the unique optimal transport map from $\mu_{t}=(X_{t})_{\#}\lambda$ to $\lambda$. 
	Moreover, from (\ref{nablaoppdist}) we obtain
	\[|\nabla f_{K_{\lambda}}(Y)|^{2}= |T \circ Y - Y|^{2}= -2 f_{K_{\lambda}}(Y),\]
	suggesting an interpretation of (\ref{MAG2}) as the Euler-Lagrange equation associated to the functional
	\begin{equation}\label{funzinterpr}
		\int_{0}^{\delta}|\dot{\gamma}|^{2}+|\nabla f_{K_{\lambda}}(\gamma)|^{2},\quad \gamma:[0,\delta]\rightarrow H_{\lambda}.
	\end{equation}
	This variational reformulation appears natural in the attempt to give a meaning to system (\ref{MAG}) also in the discrete setting, where, as it is well known, Theorem \ref{CorderoErasquin} fails.  
	\paragraph{The discrete MAG model.}
	As already mentioned before, one of the aims of this work is to go deeper in the analysis of the discrete version of the Monge-Ampère gravitational model, first introduced in \cite{Brenier2015ADL} and then formalized in \cite{AmbBrMAG}. Here we choose as reference measure 
	\[\lambda = \frac{1}{m}\overset{m}{\underset{i=1}{\sum}}\delta_{a_i},\] 
	where the $a_i$'s are distinct points on $\mathbb{T}^{n}$ (think for instance to a regular lattice approximating the uniform measure). In this case, the space $H_{\lambda}$ is easily seen to be finite dimensional, and isomorphic to $\mathbb{R}^{nm}$, through the identification of a map $Y\in H_{\lambda}$ with the $m$-uple $(Y(a_{1}),\dots,Y(a_{m}))\in \left(\mathbb{R}^{n}\right)^{m}$. Under this correspondence, $K_{\lambda}$ is represented by the discrete set of all points $(b_1,\dots,b_m)$ in $\mathbb{R}^{nm}$ such that $\pi\left(\left\{b_1,\dots,b_m\right\}\right)=\left\{a_1,\dots,a_m\right\}$. By regarding, a bit improperly, the $a_i$'s as elements of $[0,1)^{n}$, the set $K_\lambda$ can be written as the union of $m!$ cubic lattices in $\mathbb{R}^{nm}$:
	\[K_{\lambda}=\underset{\sigma \in \mathfrak{S}_m}{\bigcup}(a_{\sigma(1)},\dots,a_{\sigma(m)})+\mathbb{Z}^{nm}.\]
	In this discrete scenario, the MAG model describes the motion of $m$ particles of equal mass $1/m$ in the torus $\mathbb{T}^{n}$, whose dynamics is ruled by the optimal transport problem as follows. The position of the $i$-th particle at time $t$ is denoted by $x_{i}(t)=X_{t}(a_{i})$, and a lifting of $x_{i}(t)$ to $\mathbb{R}^{n}$ by $y_{i}(t)=Y_{t}(a_i)$. The equivalent of (\ref{MAG2}) in this setting is, at least formally, 
	\begin{equation}\label{MAGdiscreto}
		\ddot{y}_{i}(t)=y_{i}(t)-b_{i}^{\opt}(t),\quad i\in [1:m],
	\end{equation}
	where $(b_{1}^{\opt}(t),\dots,b_{m}^{\opt}(t))$ is the closest point to $(y_{1}(t),\dots,y_{m}(t))$ in $K_{\lambda}$. The system (\ref{MAGdiscreto}) is easily seen to be ill posed, because of the general non uniqueness of the projection on $K_\lambda$, ultimately due to the nonuniqueness of the optimal transport map in the discrete setting, in contrast with the absolutely continuous one. As already pointed out in \cite{AmbBrMAG}, in order to fix this problem, it is convenient to switch to a variational reformulation of the dynamical system, by considering an action functional of type (\ref{funzinterpr}). Therefore, relying on a least action principle, we say that $y\in AC([0,\delta], \mathbb{R}^{nm})$ is a solution of the discrete MAG model if it is a local minimizer of the functional 
	\[\int_{0}^{\delta}|\dot{\gamma}|^{2}+|\nabla f_{K_{\lambda}}(\gamma)|^{2}\]
	subject to endpoints constraints. In the next secion, we are going to study a more general functional in which $K_{\lambda}$ is replaced by a generic discrete set $K$ in $\mathbb{R}^{d}$. 
	
	Before concluding this part, we would like to briefly turn the attention of the reader to an analoguous Lagrangian problem in the space of probability measures $(\mathscr{P}(\mathbb{T}^{n}),W_{2})$. This can be obtained from MAG by dropping the parametric description of the gravitating matter, required by the Hilbertian setting of subsection \ref{subsecactionfunc}, and directly considering the evolution of a probability measure $\mu_{t}$ in $\mathbb{T}^{n}$. 
	
	\paragraph{A related Lagrangian problem in $(\mathscr{P}(\mathbb{T}^{n}),W_{2})$.}
	Far from being limited to the Hilbertian context, functionals of type (\ref{primafunc}) can be considered in a much more general metrict setting, provided that we interpret $|\nabla f(x)|$ as the descending slope $|\nabla^{-}f|(x)$ defined in (\ref{descendingslope}), and $|\dot{\gamma}|$ as the metric derivative of an absolutely continuous curve $\gamma:[0,1]\rightarrow X$. We avoid to repeat all the constructions in this new scenario (see \cite{AmbBre21} for a systematic introduction), and prefer to immediately specialize to our case of interest. We take $(X,d)=(\mathscr{P}(\mathbb{T}^{n}), W_{2})$ the space of probability measures on $\mathbb{T}^{n}$ endowed with the Wasserstein distance induced by the optimal transport problem with quadratic cost. As it is well known, $(X,d)$ is compact, geodesic and positively curved (see \cite{AGS}). Given a \lq\lq reference'' probability measure $\lambda \in \mathscr{P}(\mathbb{T}^{n})$, we consider the opposite squared distance function from $\lambda$, namely
	\[f_{\lambda}(\mu)=-\frac{W_{2}^{2}(\mu, \lambda)}{2}.\]
	Since $(X,d)$ is positively curved, we easily deduce that $f_{\lambda}$ is $(-1)$-convex (in the metric setting, convexity has to be intended along geodesics). Moreover, we can bound the descending slope of $f_{\lambda}$ at $\mu$ as follows:
	\begin{equation}\label{bounddescendingslope}
		|\nabla^{-}f_{\lambda}|(\mu)=\underset{\nu \rightarrow \mu}{\limsup}\,\frac{[W_{2}^{2}(\nu, \lambda)-W_{2}^{2}(\mu,\lambda)]^{+}}{2W_{2}(\mu, \nu)}\le \underset{\nu \rightarrow \mu}{\limsup}\,\frac{W_{2}(\nu,\lambda)+W_{2}(\mu,\lambda)}{2}=W_{2}(\mu,\lambda).
	\end{equation} 
	A precise characterization of $|\nabla^{-}f_{\lambda}|(\mu)$ is given in \cite[Theorem 10.4.12]{AGS}, and involves the minimal $L^{2}$ norm of the barycentric projection of optimal transport plans. 
	Inspired by the MAG model, and in particular by formulas (\ref{wassersteindist}) and (\ref{nablaoppdist}), one could study the Lagrangian problem associated to the lower semicontinuous functional $I^{\delta}_{f_{\lambda},\mu_{0},\mu_{\delta}}:C([0,\delta],X)\rightarrow [0,+\infty]$ defined by
	\[I^{\delta}_{f_{\lambda},\mu_{0},\mu_{\delta}}(\gamma)=
	\begin{cases}
		\displaystyle\int_{0}^{\delta}|\dot{\gamma}|^{2}+|\nabla^{-} f_{\lambda}|^{2}(\gamma)\quad &\text{if $\gamma \in AC([0,\delta],X)$, and $\gamma(0)=\mu_{0}$, $\gamma(\delta)=\mu_{\delta}$},\\
		+\infty  &\text{otherwise}.
	\end{cases}
	\]
	From the compactness of $X$, we immediately obtain the existence of minimizers of $I^{\delta}_{f_{\lambda},\mu_{0},\mu_{\delta}}$. In addition, by exploiting a generalization of Theorem \ref{stab} to the general metric setting provided by \cite[Theorem 17]{AmbBre21}, as well as the bound on the descending slope (\ref{bounddescendingslope}), we obtain the following stability result:
	
	\begin{prop}
		Let $\lambda_{j}, \lambda \in \mathscr{P}(\mathbb{T}^{d})$ be reference measures, and $\mu_{j,0},\mu_{j,\delta},\mu_{0},\mu_{\delta}\in \mathscr{P}(\mathbb{T}^{d})$ be endpoints. Suppose that 
		\begin{itemize}
			\item[i)] $\lambda_{j}\rightarrow \lambda$ in $W_{2}$.
			\item[ii)] $\mu_{j,i}\rightarrow \mu_{i}$ in $W_{2}$, for $i=0,\delta$.
		\end{itemize} 
		Then $I^{\delta}_{f_{\lambda_{j}},\mu_{j,0},\mu_{j,\delta}}$ $\Gamma$-converge to $I^{\delta}_{f_{\lambda},\mu_{0},\mu_{\delta}}$ in the $C([0,\delta],X)$ topology. Moreover, we have convergence of minimal values to minimal values and of minimizers to minimizers.
	\end{prop}
	\section{The case of the opposite squared distance function in $\mathbb{R}^{d}$}\label{secoppsqdistfunc}
	In this section the main results of the paper will be derived. We study functionals of type (\ref{funzionaleconendpoits}) in the special case in which $H=\mathbb{R}^{d}$ and $f=f_{K}$ is the opposite squared distance function from a closed subset $K\subseteq \mathbb{R}^{d}$. Motivated by the variational reformulation of the discrete MAG model, derived in the previous section, we will in particular focus on the case in which $K$ is a discrete collection of points in $\mathbb{R}^{d}$. In this last setting, we will exploit the geometrical structure given by the associated Voronoi decomposition of the space in order to get regularity for local minimizers out of a finite number of \lq\lq shock times''.
	
	Given a closed set $K\subseteq \mathbb{R}^{d}$, we consider the opposite squared distance function from $K$, defined by
	\begin{equation}\label{oppdistrn}
		f_{K}(x):=-\frac{\dist^{2}_{K}(x)}{2}=-\underset{y\in K}{\min}\frac{|x-y|^{2}}{2}.
	\end{equation}
	Notice that the infimum in (\ref{oppdistfunc}) is always attained here, due to the local compactness of the ambient space. The convex function 
	\begin{equation}\label{goppdistrn}
		g_{K}(x):=\underset{y\in K}{\max}\left(x\cdot y-\frac{|y|^2}{2}\right)
	\end{equation}
	satisfies $g_{K}(x)=f_{K}(x)+\frac{|x|^{2}}{2}$, thus implying the $(-1)$-convexity of $f_{K}$. We fix a potential shape $h:[0,+\infty)\rightarrow [0,+\infty)$ of class $C^{1}$ and consider the action functional $I_{f_{K}, x_{0}, x_{\delta}}^{\delta}:C([0,\delta],\mathbb{R}^{d})\rightarrow [0,+\infty]$ defined by
	\begin{equation*}
		I^{\delta}_{f_{K},x_{0},x_{\delta}}(\gamma):=
		\begin{cases}
			\displaystyle\int_{0}^{\delta}|\dot{\gamma}|^2+h(|\nabla f_{K}|^{2}(\gamma))\quad &\text{if $\gamma \in AC([0,\delta], \mathbb{R}^{d})$ and $\gamma(0)=x_{0}, \gamma(\delta)=x_{\delta}$},\\
			+\infty  &\text{otherwise}.  	
		\end{cases}
	\end{equation*} 
	
	\noindent We stress that $\nabla f_{K}$ has to be intended as an extended gradient, because $f_{K}$ is differentiable only at those points in which the projection on $K$ is unique. 
	
	In order to get a useful characterization of $\nabla f_{K}$, we need a well known Lemma of convex analysis providing an explicit formula for the subdifferential at $x$ of the maximum of a family of convex functions, under suitable assumptions (see \cite{ConvAnal}).
	\begin{lem}[Subdifferential of the sup function]
		\label{subdiffsupfunc}
		Let $\left\{g_{\alpha}:\mathbb{R}^{d}\rightarrow \mathbb{R}\right\}_{\alpha \in \mathcal{A}}$ be a collection of convex functions indexed on a compact metric space $\mathcal{A}$, and suppose that $\alpha \mapsto g_{\alpha}(x)$ is upper semicontinuous for every $x\in \mathbb{R}^{d}$. We consider the supremum function 
		\[g:=\underset{\alpha\in \mathcal{A}}{\sup}\,g_{\alpha}.\]
		Then, if the supremum in the definition of $g(x)$ is attained, the following formula holds for the subdifferential of g at $x$:
		\[
		\partial g(x)=\conv \left(\bigcup \left\{\partial g_{\alpha}(x): g_{\alpha}(x)=g(x)\right\}\right).
		\]   
	\end{lem}
	\noindent We call $\opt_{K}(x)$ the compact subset of $K$ containing all the points that minimize the distance from $x$:
	\[\opt_{K}(x):=\left\{y \in K: {\dist}_{K}(x)=|x-y|\right\}=\left\{y\in K: g_{K}(x)=x\cdot y-\frac{|y|^2}{2}\right\}.\]
	In the sequel we will refer to $\opt_{K}(x)$ as the \textit{optimality class} of $x$. Applying Lemma \ref{subdiffsupfunc} we get:
	\begin{prop}[Subdifferential of the opposite squared distance function]
		\label{propsuboppdist}
		Let $K\subseteq \mathbb{R}^{d}$ be a closed set, and let $f_K, g_K$ be defined as in (\ref{oppdistrn}) and (\ref{goppdistrn}). Then
		\begin{itemize}
			\item[i)] The subdifferential of $g_K$ at $x$ is given by
			\[\partial g_{K}(x)=\conv(\opt_{K}(x)).\]
			\item[ii)] The subdifferential of $f_{K}$ at $x$ is given by
			\[\partial f_{K}(x)=\conv(\opt_{K}(x))-x.\]
			Moreover, denoting by $\eta_{K}(x)$ the unique projection of $x$ on the closed convex set $\conv(\opt_{K}(x))$, the following formula holds for the extended gradient of $f_{K}$ at $x$:
			\begin{equation}\label{extgroppdistfunc}
				\nabla f_{K}(x)=\eta_{K}(x)-x.
			\end{equation}
			\item[iii)] The point $\eta_{K}(x)$ depends only on the optimality class of $x$. That is to say, $\eta_{K}(x)=\eta_{K}(y)$, whenever $\opt_{K}(x)=\opt_{K}(y)$.
		\end{itemize}  
	\end{prop} 
	\begin{proof}
		Point i) easily follows from Lemma \ref{subdiffsupfunc} if $K$ is compact. To deal with the general case it is enough to notice that for every $x\in \mathbb{R}^{d}$ and every radius $R>\dist_{K}(x)$, we have $g_{K}=g_{K\cap \overline{B_{R}(x)}}$ in a neighborhood of $x$. The formula for $\partial f_{K}(x)$ is a consequence of the rule for the subdifferential of the sum of two functions, one of which smooth. Then, by definition, $\nabla f_{K}(x)$ is the projection of $0$ on $\partial f_{K}(x)=\conv(\opt_{K}(x))-x$, and formula (\ref{extgroppdistfunc}) follows after a translation of $x$. Let us now address point iii). Suppose that $x$ and $y$ share the same optimality class, $\opt_{K}(x)=\opt_{K}(y)$. Consider the affine space $A$ spanned by $\opt_{K}(x)$ and its orthogonal space $B$ passing through $x$. From the hypothesis on $x$ and $y$ we deduce that also $y$ belongs to $B$. Then, denoting by $p$ the point of intersection of $A$ and $B$, by orthogonality we have:
		\[|x-z|^2=|x-p|^2+|p-z|^2,\quad |y-z|^2=|y-p|^2+|p-z|^2\quad \text{for every $z\in {\conv}(\opt_{K}(x)).$}\]
		Hence, both distances are minimized by the point $z$ obtained by projecting $p$ on ${\conv}(\opt_{K}(x))$, thus $\eta_{K}(x)=\eta_{K}(y)$.
	\end{proof}
	\begin{oss}
		From (\ref{extgroppdistfunc}) we deduce that the potential term $|\nabla f_{K}(x)|^2$ is always less than or equal to $-2f_{K}(x)=\dist_{K}^{2}(x)$, and equality holds if and only if $x$ has a unique projection on $K$. It is interesting to see what this means for the MAG model. Using the notation of subsection \ref{subsectionMAG}, the potential of a configuration $(y_{1},\dots,y_{m})\in \mathbb{R}^{nm}$ is always smaller than $W_{2}(\mu,\lambda)$, where, setting $x_i = \pi(y_{i})$,
		\[\lambda = \frac{1}{m}\overset{m}{\underset{i=1}{\sum}}\delta_{a_i},\quad 	\mu = \frac{1}{m}\overset{m}{\underset{i=1}{\sum}}\delta_{x_i}
		\]
		and $W_{2}$ is the Wasserstein distance in $\mathscr{P}_{2}(\mathbb{T}^{n})$. Moreover, equality holds if and only if there exists a unique optimal transport map from $\mu$ to $\lambda$. So we see that in the context of the MAG model, the potential term should be interpreted as a \lq\lq measure of the ambiguity in the optimal transport problem''. Tipical manifestations of ambiguity in the discrete scenario appear when two or more particles collapse, thus sharing the same position in $\mathbb{T}^{n}$. Compare also this phenomenon with the continuous framework of Theorem \ref{CorderoErasquin}, where this ambiguity does not occur.\fr 
	\end{oss}
	To end this part, we briefly come back to the matter of stability in this specialized context, stating the following Corollary of Theorem \ref{stab}:
	\begin{cor}
		\label{stab2}
		Let $K_{j}, K$ be closed subsets of $\mathbb{R}^{d}$, and let $x_{j,0}, x_{j,\delta}, x_{0}, x_{\delta}\in \mathbb{R}^{d}$. Suppose that 
		\begin{itemize}
			\item[i)] $K_{j}\rightarrow K$ in the sense of Hausdorff in every compact set.
			\item[ii)] $x_{j,i}\rightarrow x_{i}$ for $i= 0,\delta$.
		\end{itemize}
		Then $I^{\delta}_{f_{K_{j}},x_{j,0},x_{j,\delta}}$ $\Gamma$-converge to $I^{\delta}_{f_{K},x_{0},x_{\delta}}$ in the $C([0,\delta],\mathbb{R}^{d})$ topology. Moreover, we have convergence of minimal values to minimal values and of minimizers to minimizers.
	\end{cor}
	\noindent As a consequence, the functional associated to a closed set $K$ can be approximated by functionals associated to $K_{j}$, where each $K_{j}$ is a finite collection of points in $\mathbb{R}^{d}$. It is the scope of the following subsection to focus on this simpler situation.     
	
	\subsection{The discrete case}
	From now on, we restrict our analysis to the case in which $K$ is given by a collection of $N$ distinct points in $\mathbb{R}^{d}$:
	\[K=\left\{p_{1},\dots,p_{N}\right\}.\]
	We are particularly interested in studying properties of local minimizers for $I^{\delta}_{f_{K},x_{0},x_{\delta}}$ because of the link with the variational reformulation of the discrete MAG model (see the discussion above). There, $K$ was an infinite discrete set, but we can clearly restrict our analysis, which is essentially local, to the case in which $K$ is finite, due to the compactness of the range of every continuous curve $\gamma:[0,\delta]\rightarrow \mathbb{R}^{d}$. Let us fix $K$, so as to be allowed to omit all the pedices involving it. Then, for instance, we will write $f, g, \eta, \opt$ in the place of $f_{K}, g_{K}, \eta_{K}, \opt_{K}$. 
	
	\paragraph{Polyhedra, Voronoi cells and potential zones.}
	We say that $P\subseteq \mathbb{R}^{d}$ is a \textit{polyhedron} if it is a nonempty closed convex set admitting a representation of the form
	\begin{equation}\label{rapprepoliedri}
		P=\overset{\ell}{\underset{j=1}{\bigcap}}\left\{x\in \mathbb{R}^{d}: T_{j}(x)\le 0\right\},
	\end{equation}
	where $\ell \in \mathbb{N}$ and $T_{j}:\mathbb{R}^{d}\rightarrow \mathbb{R}$ are affine functions. A bounded polyhedron is called a \textit{polytope}.
	
	The \textit{Voronoi partition} associated to a finite collection of points $K=\left\{p_{1},\dots,p_{N}\right\}$ is the finite decomposition $\left\{V_{H}\right\}_{H\in \mathcal{P}(K)}$ of $\mathbb{R}^{d}$, indexed by the set $\mathcal{P}(K)$ of the parts of $K$, and such that 
	\[V_{H}=\left\{x\in \mathbb{R}^{d}: \opt(x)=H\right\}.\] 
	We call $V_H$ the \textit{Voronoi cell} corresponding to the \textit{optimality class} $H$. The following are well known facts about this remarkable cellular decomposition of the space (see \cite{Voronoi}). 
	\begin{prop}[Properties of the Voronoi partition]
		Let $H\in \mathcal{P}(K)$ be such that the Voronoi cell $V_{H}$ is nonempty. Then
		\begin{itemize}
			\item[i)] $V_{H}$ is a convex set. Moreover, denoting by $A_H$ the affine space spanned by $H$, and by $B_H$ the affine space 
			\[B_{H}:=\left\{x\in \mathbb{R}^{d}: |x-p_i|=|x-p_j|, \, \forall p_i, p_j \in H\right\},\]
			we have that $A_H$ is orthogonal to $B_H$, they have complementary dimensions in $\mathbb{R}^{d}$, and $V_H$ is relatively open in $B_H$. We call $p_H$ the unique intersection point of $A_H$ and $B_H$.
			\item[ii)] The closure $\overline{V_{H}}$ is a polyhedron, whose relative boundary in $B_H$ is precisely given by the disjoint union of all the Voronoi cells $V_{L}$ with $H\subsetneq L$.    
		\end{itemize}
	\end{prop}
	In order to introduce the second fundamental decomposition associated to $K$, we also need to define, for every $\eta \in \mathbb{R}^{d}$, the sets:
	\begin{align*}
		Q_{\eta}&:=\left\{x\in \mathbb{R}^{d}: \eta(x)=\eta\right\},\\
		P_{\eta}&:=\left\{x\in \mathbb{R}^{d}: \eta \in \partial g(x)\right\}.
	\end{align*}
	The following proposition encodes the underlying geometrical structure conferred to our variational problem by the particular choice we made for the potential. It will be of fundamental importance in deriving regularity results for local minimizers of the functional $I^{\delta}_{f_{K},x_{0},x_{\delta}}$.
	\begin{prop}[Voronoi cells and potential zones]
		\label{vorcellsandpotzones}
		The following facts hold:
		\begin{itemize}
			\item[i)] The map $\eta$ is constant in each Voronoi cell, and hence has a finite range, that we denote by $\mathcal{E}$.
			\item[ii)] $\left\{Q_{\eta}\right\}_{\eta \in \mathcal{E}}$ is a partition of $\mathbb{R}^{d}$, and $x\in Q_{\eta}$ if and only if $\nabla f(x)=\eta-x$. In the sequel we will refer to the $Q_{\eta}$'s as \textit{potential zones}.
			\item[iii)] For every $\eta \in \mathcal{E}$, both $Q_{\eta}$ and $P_{\eta}$ are union of Voronoi cells and it holds $Q_{\eta} \subseteq P_{\eta}$.
			\item[iv)] For every $\eta \in \mathcal{E}$, $P_{\eta}$ is a polyhedron.
			\item[v)] Let $\beta$ be the positive constant defined by
			\begin{equation}\label{minimaldeviation}
				\beta:=\underset{\eta, \bar{\eta} \in \mathcal{E}}{\underset{\eta \neq \bar{\eta}}{\min}}|\eta-\bar{\eta}|^2.
			\end{equation}
			Then we have
			\begin{equation}\label{separazionelivellipotenziale}
				|\bar{\eta}-x|^2\ge |\eta-x|^2+\beta \quad \text{for every distinct $\eta, \bar{\eta} \in \mathcal{E}$ and $x\in Q_{\eta}\cap P_{\bar{\eta}}$.}
			\end{equation}
		\end{itemize}
	\end{prop}
	\begin{proof}
		Point i), ii) and iii) are direct consequences of Proposition \ref{propsuboppdist}. To prove point iv) it is enough to notice that $P_{\eta}$ is a closed convex set that can be written as the union of a finite number of polyhedra (the closures of the Voronoi cells contained in $P_{\eta}$). Finally, point v) easily follows from the fact that $\eta$ is the projection of $x$ on the closed convex set $\partial g(x)$ containing $\bar{\eta}$.
	\end{proof}
	So we see that $K$ carries two partitions of $\mathbb{R}^{d}$, one finer than the other: the first into Voronoi cells and the second into potential zones. Simple examples show that for a general $K$ they do not coincide (see for instance Example \ref{ex2} hereafter). If they coincide, we say that $K$ is \textit{balanced}. In such a case, the map $\eta$ defines a bijection between Voronoi cells and potential zones, that is to say, 
	\begin{equation*}
		\eta(x)=\eta(y)\iff \opt(x)=\opt(y).
	\end{equation*} 
	Clearly, a sufficient condition for $K$ to be balanced is given by
	\begin{equation}\label{hypperregvor}
		\opt(\eta(x))=\opt(x)\quad \text{for every $x\in \mathbb{R}^{d}$.}
	\end{equation} 
	It is worth noting that in dimension $d=1$ every $K$ is balanced, and that the same is true in any dimension for cubic lattices.
	\paragraph{Conserved quantities.}
	In this paragraph we underline the presence of some conserved quantities for local minimizers of our variational problem. They naturally arise by testing the local minimality against variations along some specific directions. The following Lemma collects two crucial observations in order to suitably perform such variations:
	\begin{lem}[Local properties of the Voronoi diagram]
		\label{locpropvordiag}
		The following facts hold:
		\begin{itemize}
			\item[i)] If $x\in V_H$, for some $H\subseteq K$, then there exists a neighborhood $U$ of $x$ such that $\opt(y)\subseteq H$, for every $y\in U$.
			\item[ii)] If $x\in V_H$, then $x+\epsilon v \in V_H$, provided that the vector $v$ is parallel to $B_H$, and $\epsilon$ is sufficiently small.	
		\end{itemize}
	\end{lem}
	\begin{proof}
		Point i) follows from the fact that a point of $K$ which is not optimal for $x$ is neither optimal for $y$, provided that $y$ is chosen close enough to $x$. Point ii) is instead a direct consequence of the fact that $V_{H}$ is relatively open in the affine space $B_H$.
	\end{proof} 
	\noindent By using very classical variational arguments as well as Lemma \ref{locpropvordiag} we derive the following 
	\begin{prop}[Conservation laws]
		\label{conserv}
		Let $\gamma$ be a local minimizer of $I^{\delta}_{f_{K},x_{0},x_{\delta}}$. Then 
		\begin{itemize}
			\item[i)](Conservation of the energy). There exists a constant $c\in \mathbb{R}$ such that
			\[|\dot{\gamma}|^{2}-h(|\nabla g(\gamma)|^{2})=c\quad \text{a.e.\ in $(0,\delta)$.}\] 
			In particular, $\gamma$ is Lipschitz continuous.
			\item[ii)](Local conservation of momentum). Let $(t_1, t_2\subseteq [0,\delta]$ be a time interval. Suppose that there exists an optimality class $H\subseteq K$, with $V_H \neq \emptyset$, such that for every $s\in (t_1, t_2)$ the inclusion $\opt(\gamma(s))\subseteq H$ holds. Let $\gamma_H$ be the curve obtained by projecting $\gamma$ on the affine space $B_H$. Then $\gamma_H$ is a $C^{1,1}$ curve in $(t_1, t_2)$ and, in this interval, it satisfies 
			\[\ddot{\gamma}_H=h'(|\gamma-\eta(\gamma)|^2)(\gamma_H - p_H).\]
			In particular, for each time $t\in(0,\delta)$, denoting $H=\opt(\gamma(t))$, there exists a neighborhood of $t$ in which the component $\dot{\gamma}_H$ of the momentum parallel to $B_H$ is continuous.
		\end{itemize}
	\end{prop}
	\begin{proof}
		Point i) states that $\gamma$ solves the Dubois-Reymond equation (\ref{DBR}). This can be shown by testing the local minimality through \lq\lq horizontal'' variations of the independent variable of the form $\gamma_{\epsilon}= \gamma \circ \rho_{\epsilon}^{-1}$, where $\rho_{\epsilon}=\id + \epsilon \varphi$, $\varphi \in C^{\infty}_{c}((0,\delta))$, and $\epsilon$ is small enough so that $\rho_{\epsilon}$ is a diffeomorphism. To get point ii), instead, we need to perform \lq\lq vertical'' variations of the form $\gamma_{\epsilon} = \gamma + \epsilon \varphi v$, where $\varphi \in C^{\infty}_{c}((t_1,t_2))$ and $v$ is any vector parallel to the affine space $B_{H}$. We then use point ii) of Lemma \ref{locpropvordiag} to get $\eta(\gamma_{\epsilon})=\eta(\gamma)$, for $\epsilon$ sufficiently small.
	\end{proof}
	\noindent Point ii) of the previous proposition implies in particular that $\gamma$ is regular as long as it stays in a single Voronoi cell. More precisely:
	\begin{cor}[Regularity inside a Voronoi cell]
		\label{reginvoronoicell}
		Let $\gamma$ be a local minimizer of $I^{\delta}_{f_{K},x_{0},x_{\delta}}$. Suppose that $\opt(\gamma(s))=H$ is constant in $(t_1, t_2)$ and let $\eta_H = \eta(V_{H})$. Then, in this interval, $\gamma=\gamma_H$,  $\eta(\gamma)=\eta_H$ is constant, and $\gamma$ is a $C^{2}$ solution of
		\[\ddot{\gamma}=h'(|\gamma-\eta_H|^2)(\gamma - p_H).\]
		Moreover, if h is $C^{\infty}$, then $\gamma \in C^{\infty}((t_1,t_2))$.
	\end{cor}
	As a natural consequence, any singularity for a local minimizer of the functional $I^{\delta}_{f_{K},x_{0},x_{\delta}}$ appears only when the optimality class \lq\lq changes''. In the next paragraph we will try to give a more precise meaning to this statement. 
	\paragraph{Shock times and minimal deviation.}
	Given a curve $\gamma:[0,\delta]\rightarrow \mathbb{R}^{d}$, and a time $t\in [0,\delta]$, we say that
	\begin{itemize}
		\item $t$ is a \textit{shock} for $\gamma$ if $\opt(\gamma)$ is not constant on $I$ for every neighborhood $I$ of $t$ in $[0,\delta]$.
		\item $t$ is a \textit{nondegenerate shock} for $\gamma$ if $\eta(\gamma)$ is not constant on $I$ for every neighborhood $I$ of $t$ in $[0,\delta]$.
		\item $t$ is an \textit{effective shock} for $\gamma$ if there are two distinct potential zones $Q_{\eta}, Q_{\bar{\eta}}$, and two distinct Voronoi cells $V_{H}\subseteq Q_{\eta}$, $V_{\bar{H}}\subseteq Q_{\bar{\eta}}$ such that, for some $\epsilon >0$, one of the following holds:
		\begin{itemize}
			\item[-] $\gamma((t-\epsilon, t))\subset V_{H}$ and $\gamma([t, t+\epsilon))\subset V_{\bar{H}}$. In this case, $H\subsetneq \bar{H}$ and we say that $t$ is a \textit{left effective shock}.
			\item[-] $\gamma((t-\epsilon, t])\subset V_{H}$ and $\gamma((t, t+\epsilon))\subset V_{\bar{H}}$. In this case, $\bar{H}\subsetneq H$ and we say that $t$ is a \textit{right effective shock}.
		\end{itemize}
	\end{itemize}
	We denote by $S(\gamma)$, $NDS(\gamma)$ and $ES(\gamma)$ respectively the sets of shocks, nondegenerate shocks and effective shocks for $\gamma$. Notice that $ES(\gamma)\subseteq NDS(\gamma)\subseteq S(\gamma)$, and that $S(\gamma)$ and $NDS(\gamma)$ are compact. According to the definitions above, during a shock there must be a change of Voronoi cell, while, during a nondegenerate shock there is also a change of potential zone. Clearly we have $S(\gamma)=NDS(\gamma)$ provided that $K$ is balanced. Finally, we have an effective shock when a neat passage occurs from a Voronoi cell to an adjacent one with different potential. By conservation of the energy, we expect the dynamics to develop a singularity in the kinetic term here. This is the content of the following proposition, which is a direct consequence of the conservation laws stated in Proposition \ref{conserv}.
	\begin{prop}[Minimal deviation during an effective shock]
		Suppose that the potential shape $h$ is strictly increasing.
		Let $\gamma$ be a local minimizer of $I^{\delta}_{f_{K},x_{0},x_{\delta}}$. Suppose that $t\in (0,\delta)$ is a left effective shock for $\gamma$, in which $\gamma$ jumps from $V_H$ to $V_{\bar{H}}$. Let $\eta_H = \eta(V_{H})$ and $\eta_{\bar{H}}=\eta(V_{\bar{H}})$. Defined $\dot{\gamma}_{-}(t)$ and $\dot{\gamma}_{+}(t)$ respectively by
		\[\dot{\gamma}_{-}(t):=\underset{s\rightarrow t^{-}}{\lim}\dot{\gamma}(s)\quad\quad \dot{\gamma}_{+}(t):=\underset{s\rightarrow t^{+}}{\lim}\dot{\gamma}(s),\]
		then $\dot{\gamma}_{+}(t)$ is the component of $\dot{\gamma}_{-}(t)$ parallel to $B_{\bar{H}}$. Moreover
		\[|\dot{\gamma}_{-}(t)-\dot{\gamma}_{+}(t)|^2=|\dot{\gamma}_{-}(t)|^2-|\dot{\gamma}_{+}(t)|^2=h(|\gamma(t)-\eta_H|^2)-h(|\gamma(t)-\eta_{\bar{H}}|^2)>0\] 
	\end{prop}
	\noindent Clearly an analogous result holds for right effective shocks. 
	\begin{oss}
		Notice that, in the specific case of a superadditive shape $h$, (for instance when $h=\id$, as in the MAG model), we derive a uniform lower bound on the jump of the derivative during an effective shock:
		\[|\dot{\gamma}_{-}(t)-\dot{\gamma}_{+}(t)|^2 = h(|\gamma(t)-\eta_H|^2)-h(|\gamma(t)-\eta_{\bar{H}}|^2)\ge h(|\eta_H-\eta_{\bar{H}}|^2)\ge h(\beta),\]
		where $\beta$ was defined in (\ref{minimaldeviation}).\fr
	\end{oss}  
	\begin{oss}
		In the MAG dynamics, shocks tipically happen when two or more particles collide or separate, generating an instant change in the optimality class. For instance an effective shock occurs when two particles collide and remain sticked together. Notice that Proposition \ref{conserv} tells us that energy and momentum are conserved in a collision. \fr
	\end{oss}
	To end this part, we show, through a couple of simple examples, that all of the three types of shock times defined in this paragraph may occur for a minimizer of $I^{\delta}_{f_{K},x_{0},x_{\delta}}$. 
	
	\begin{ex}[Effective and noneffective shocks]\label{ex1}
		Take $h=\id$, $\delta=1$, $d=1$, $K=\left\{-1, 1\right\}$, $x_{0}=-c$, $x_{1}=c$, with $c\ge 0$. We see that $\mathbb{R}$ is partitioned into three Voronoi cells, namely: $V_{\left\{-1\right\}}=(-\infty,0)$, $V_{\left\{1\right\}}=(0,+\infty)$ and $V_{\left\{-1, 1\right\}}=\left\{0\right\}$. The potential term takes the form 
		\[|\nabla f(x)|^2 =
		\begin{cases}
			|x+1|^2\quad &\text{if $x\in (-\infty,0)$},\\
			0  &\text{if $x=0$},\\
			|x-1|^2\quad &\text{if $x\in (0,+\infty)$}.
		\end{cases}
		\]
		It is easily seen that a minimizer $\gamma$ of $I^{\delta}_{f_{K},x_{0},x_{\delta}}$ has to be nondecreasing, and thus $\gamma^{-1}(0)=[t_0, t_1]$ is a closed interval (possibly degenerate if $t_0 = t_1$), and $S(\gamma)=NDS(\gamma)=\left\{t_0, t_1\right\}$. Then, there are two possible qualitatively different behaviours of $\gamma$, according to whether $t_0 = t_1$ or not. If $t_0 = t_1$, then $\gamma$ has a single nondegenerate non effective shock time. If instead $t_0 \neq t_1$, then $t_0$ and $t_1$ are respectively left and right effective shocks. Now, direct computations show that the first case occurs if for instance $c=1$. On the other hand, we can prove that the second case necessarily occurs if $c$ is chosen sufficiently small. As a matter of fact, for $c<1$, by the monotonicity of $\gamma$, the minimum value $\Gamma(-c,c)$ of the functional can be bounded below as follows:
		\[\Gamma(-c,c)\ge (1-t_1 + t_0)(1-c)^2.\]
		At the same time, by Corollary \ref{stab2}, we have
		\[\Gamma(-c,c)\rightarrow \Gamma (0,0) = 0,\quad \text{as $c\rightarrow 0^{+}$}.\]
		Therefore, $t_1>t_0$, provided that we choose $c$ small enough.
	\end{ex}   
	
	\begin{ex}[Degenerate shocks]\label{ex2}
		Take $h=\id$, $\delta = 1$, $d=2$, $K=\left\{(1,0), (0,1), (-1,0)\right\}$, $x_0 = (0,-1)$, $x_{1}=(0,0)$. Here we notice that the two distinct Voronoi cells $V_{\left\{(1,0), (0,1), (-1,0)\right\}}=\left\{(0,0)\right\}$ and $V_{\left\{(1,0), (-1,0)\right\}}=\left\{0\right\}\times (-\infty,0)$ share the same potential zone $Q_{\eta}=\left\{0\right\}\times (-\infty,0]$ with $\eta = (0,0)$. Then, by comparison with the projection over $Q_{\eta}$, we see that every minimizer $\gamma$ of $I^{\delta}_{f_{K},x_{0},x_{\delta}}$ must live in $Q_\eta$, thus the first time $t$ at which $\gamma(t)=(0,0)$ has to be a degenerate shock for $\gamma$.  
	\end{ex}
	
	\subsection{Regularity results}
	Here we state and prove our main regularity results. Recall that $K=\left\{p_1,\dots,p_N\right\}$ is a finite collection of points in $\mathbb{R}^{d}$, and $h:[0,+\infty)\rightarrow [0,+\infty)$ is a $C^{1}$ potential shape.  
	\begin{thm}[$C^{1,1}$ regularity out of a finite number of nondegenerate shock times]
		\label{teoprinc}   
		Let $\gamma$ be a local minimizer of $I^{\delta}_{f_{K},x_{0},x_{\delta}}$. Suppose that $h$ is strictly increasing. Then:
		\begin{itemize}
			\item[i)] The set $NDS(\gamma)$ of non degenerate shock times of $\gamma$ is finite. That is to say, there is a finite number of times $0\le t_1 <\dots<t_{\ell}\le \delta$ such that $\eta(\gamma)$ is constant in each connected component of $[0,\delta]\setminus \left\{t_1,\dots,t_{\ell}\right\}$. 
			\item[ii)] Setting $t_0 = 0$ and $t_{\ell+1}=\delta$, then $\gamma$ is $C^{1,1}$ regular in the interval $[t_i, t_{i+1}]$ for every $i\in [0:\ell]$. Moreover, if we let $\eta_i \in \mathcal{E}$ be such that $\gamma((t_i, t_{i+1}))\subseteq Q_{\eta_i}$, then we can estimate
			\[|\ddot{\gamma}(r)|\le h'(|\gamma(r)-\eta_i|^2)|\gamma(r)-\eta_i|\quad \text{for a.e.\ $r\in (t_i,t_{i+1})$}.\]
		\end{itemize} 	 
	\end{thm}
	\noindent Actually, if $K$ is balanced and $h$ is smooth, Theorem \ref{teoprinc} can be improved to reach piecewise smooth regularity for any local minimizer $\gamma$. In fact, since $K$ is balanced, the equality $NDS(\gamma)=S(\gamma)$ holds and point i) of Theorem \ref{teoprinc} implies that $\gamma$ has a finite number of shock times. On the other hand, Corollary \ref{reginvoronoicell} together with the smoothness of $h$ ensures that $\gamma$ is smooth in each connected component of $[0,\delta]\setminus S(\gamma)$, where clearly $\opt(\gamma(t))$ is constant. 
	\begin{cor}[$C^{\infty}$ regularity out of a finite number of shock times]
		\label{improvedreg}
		Let $\gamma$ be a local minimizer of $I^{\delta}_{f_{K},x_{0},x_{\delta}}$. Suppose that $K$ is balanced and $h$ is $C^{\infty}$ and strictly increasing. Then:
		\begin{itemize}
			\item[i)] The set $S(\gamma)$ of shock times of $\gamma$ is finite. That is to say, there is a finite number of times $0\le t_1 <\dots<t_{\ell}\le \delta$ such that the optimality class $\opt(\gamma)$ is constant in each connected component of $[0,\delta]\setminus \left\{t_1,\dots,t_{\ell}\right\}$. 
			\item[ii)] Setting $t_0 = 0$ and $t_{\ell+1}=\delta$, then $\gamma$ is $C^{\infty}$ in the interval $[t_i, t_{i+1}]$ for every $i\in [0:\ell]$. Moreover, if $t_{i+1}>t_{i}$, denoting by $H_i$ the optimality class of $\gamma$ in the interval $(t_i, t_{i+1})$, and defining $\eta_{H_i}=\eta(V_{H_i})$, then $\gamma$ solves
			\[\ddot{\gamma}=h'(|\gamma-\eta_{H_i}|^2)(\gamma - p_{H_i})\]
			in $(t_i, t_{i+1})$.
		\end{itemize} 	 
	\end{cor}
	\begin{oss}
		Corollary \ref{improvedreg} offers a generalization of \cite[Theorem 13]{AmbBrMAG}, in which smooth regularity out of a finite number of shock times was derived for solutions of a $1$-dimensional version of the discrete MAG model. In their framework, $h$ was simply the identity, while $K$ consisted of all the $m!$ points of $\mathbb{R}^{d}$ obtainable by permuting the components of a fixed vector $A=(a_1,\dots,a_d)\in \mathbb{R}^{d}$, with $a_1<\dots<a_d$. Exploiting the rearrangement inequality provided by the order structure of $\mathbb{R}$, it is not difficult to show that hypothesis (\ref{hypperregvor}) holds in this case, therefore implying that $K$ is balanced.\fr    
	\end{oss}
	Let us now give an overview of the proof of Theorem \ref{teoprinc} before entering in the details of the forthcoming paragraphs. We start discussing how point ii) can be derived from point i). We need to show that if $\eta(\gamma)=\eta$ is constant in an interval $(s,t)$, then $\gamma$ is $C^{1,1}$ regular in $[s,t]$. We clearly have that $\gamma([s,t])\subset P_{\eta}$. Then we observe that for every absolutely continuous curve $\rho:[s,t]\rightarrow P_{\eta}$, due to the monotonicity of $h$ and inequality (\ref{separazionelivellipotenziale}), the following inequality holds:
	\begin{equation}\label{comparison}
		\int_{s}^{t}|\dot{\rho}|^2+ h(|\nabla f|^{2}(\rho)) \le \int_{s}^{t}|\dot{\rho}|^2+h(|\eta - \rho|^2).
	\end{equation}
	Then, by a comparison argument, knowing that $\gamma$ is a local minimizer of the functional on the left hand side of (\ref{comparison}), and that $\rho = \gamma$ saturates the inequality, we get that $\gamma$ is a local minimizer for the functional on the right hand side of (\ref{comparison}), if one restricts to curves living in the closed convex set $P_\eta$. On the other hand, we are going to prove in Lemma \ref{regvinc} that any such constrained local minimizer is necessarily $C^{1,1}$ regular. We now pass to the proof of point i) of the Theorem. First notice that we can restrict ourselves to prove the following equivalent local statement:
	\begin{equation}\label{claim}
		\text{\textbf{Claim}:\quad for each time $t\in (0,\delta]$, there exists $\epsilon>0$ such that $\eta(\gamma)$ is constant in $(t-\epsilon, t)$.}
	\end{equation}
	\noindent From claim (\ref{claim}) and the analogous one for right intervals (obtainable by exploiting the autonomicity of the functional), it clearly descends that $NDS(\gamma)$ is discrete. Then, by compactness, we derive that $NDS(\gamma)$ is finite. The proof of claim (\ref{claim}) will be elementary accomplished by a quite intricated series of \lq\lq cut and paste'' constructions of competitors. Let us first outline the general heuristic idea. We divide the proof in a few steps: 
	
	\noindent \textbf{\underline{Step 1}}. Suppose by contradiction that $t$ is a cluster point for the \lq\lq jumps'' in the potential. Then, approaching $t$ from the left, $\gamma$ would infinitely often visit high and low potential zones. If only one low potential zone were visited asymptotically, then it would be convenient to stay in it definitely. Therefore we can assume that $\gamma$ infinitely often visits at least two different low potential zones, approaching $t$ from the left. Moreover, in alternating between different low potential zones, $\gamma$ necessarily spends a non negligible amount of time in high potential ones. 
	
	\noindent \textbf{\underline{Step 2}}. Approaching $t$ from the left, the percentage of time spent by $\gamma$ in high potential zones tends to zero, thus enforcing at least two different low potential zones to be very near to each other in order to make it possible for the Lipschitz curve $\gamma$ to jump from one to the other in a short time.
	
	\noindent \textbf{\underline{Step 3}}. By Lemma \ref{distpol} we will be eventually allowed to slightly deviate $\gamma$ to reach an even lower potential zone (the interface between the two), hence contradicting the local optimality of $\gamma$ and reaching the desired absurdum. 
	
	In the following paragraph we prove Lemmas \ref{regvinc} and $\ref{distpol}$. As a byproduct of Lemma \ref{regvinc}, we obtain that a local minimizer is $C^{1,1}$ regular as long as it stays in a single potential zone (see Corollary \ref{regpotzone}), hence addressing point ii) of Theorem \ref{teoprinc}. Subsequently, in the final paragraph, we will provide the rigorous proof of claim \ref{claim} along the lines of the heuristics above, thus concluding the proof of Theorem \ref{teoprinc}. 
	\paragraph{Two Lemmas.} The first Lemma, of indpendent interest, concerns the regularity of local minimizers of action functionals restricted to curves living in a given closed convex set. Because of boundary effects, such constrained minimizers are in general not $C^{2}$, even if the Lagrangian is smooth. Nevertheless, they must be at least $C^{1,1}$, whenever the Lagrangian is $C^{1}$.     
	\begin{lem}[Regularity for a problem with a convex constraint]
		\label{regvinc}
		Let $P\subseteq \mathbb{R}^{d}$ be a closed convex set, and let $\Psi:\mathbb{R}^{d}\rightarrow [0,+\infty)$ be of class $C^1$ in a neighborhood of $P$. Given two points $x_0, x_\delta \in P$, we consider the functional $\mathcal{G}:C([0,\delta],\mathbb{R}^d)\rightarrow [0,+\infty]$ defined by
		\[\mathcal{G}(\gamma)=
		\begin{cases}
			\displaystyle\int_{0}^{\delta}|\dot{\gamma}|^2+\Psi(\gamma)\quad &\text{if $\gamma \in AC([0,\delta],\mathbb{R}^d), \gamma(0)=x_0, \gamma(\delta)=x_\delta, \gamma([0,\delta])\subseteq P$},\\
			+\infty  &\text{otherwise}.
		\end{cases}
		\]
		Then every local minimizer $\gamma$ of $\mathcal{G}$ is of class $C^{1,1}$ and we can estimate
		\[|\ddot{\gamma}(r)|\le \frac{1}{2}|\nabla \Psi (\gamma(r))|\quad \text{for a.e.\ $r\in (0,\delta)$}.\] 
	\end{lem}
	\begin{proof}
		We call $\pi_P$ the projection on $P$.
		We start by defining, for each point $x\in P$, the \lq\lq blow up'' of $P$ at $x$, namely the closed convex cone $P_{x}$ defined by the formula 
		\[
		P_{x}+x=
		\begin{cases}
			\mathbb{R}^{d}	&\text{if $x\in \mathring{P}$},\\
			\bigcap\left\{\mathcal{H} \text{ half space}: P\subset \mathcal{H},\, x\in \partial \mathcal{H}\right\}\quad &\text{if $x\in \partial P$}.
		\end{cases}
		\]
		We then call $S_{x}$ the projection on $P_{x}$, which turns out to be positive homogeneous and $1$-Lipschitz. It is not difficult to realize that 
		\begin{equation}\label{Hausconv}
			\frac{P-x}{\epsilon}\xrightarrow{\epsilon\rightarrow 0^{+}}P_{x}
		\end{equation}
		in the sense of Hausdorff on every compact set. Since the inclusion $(P-x)/\epsilon \subseteq P_x$ always holds, in order to show the convergence (\ref{Hausconv}), we only need to check, using a separation argument, that, for every $\epsilon_j \rightarrow 0^{+}$, and every point $z\in P_x$, there exists a sequence $y_j \in (P-x)/\epsilon_j$ converging to $z$. Hence, the projection on $(P-x)/\epsilon$ pointwise converges to $S_x$ as $\epsilon \rightarrow 0^{+}$. By exploiting also the homogeneity of $S_x$ we eventually obtain that, for every $v\in
		\mathbb{R}^{d}$,
		\begin{equation}\label{pointwiseexpansion}
			\pi_{P}(x+\epsilon v)=x+\epsilon S_{x}(v)+o(\epsilon),\quad \text{ as $\epsilon \rightarrow 0^{+}$}.
		\end{equation}
		Now, let $\gamma$ be a local minimizer for the functional $\mathcal{G}$. Given a test function $\varphi \in C^{\infty}_{c}((0,\delta); \mathbb{R}^{d})$, we consider the following competitors:
		\[\gamma_{\epsilon}:=\pi_{P}(\gamma+\epsilon \varphi),\quad \text{for $\epsilon>0$}.\]
		Since $\pi_{P}$ is $1$-Lipschitz, we have $|\gamma_{\epsilon}'|\le |\gamma'+\epsilon \varphi'|$. Then
		\[\mathcal{G}(\gamma_{\epsilon})\le \mathcal{G}(\gamma)+2\epsilon \int_{0}^{\delta}\gamma'\cdot \varphi'+\frac{\Psi(\gamma_{\epsilon})-\Psi(\gamma)}{2\epsilon}+\epsilon^{2}\int_{0}^{\delta}|\varphi'|^2.\]
		Using the previous pointwise expansion (\ref{pointwiseexpansion}), as well as the dominated convergence Theorem, we obtain
		\begin{equation}\label{positivoRHS}
			\underset{\epsilon \rightarrow 0^{+}}{\limsup}\,\frac{\mathcal{G}(\gamma_{\epsilon})-\mathcal{G}(\gamma)}{2\epsilon}\le \int_{0}^{\delta}\gamma'\cdot \varphi' + \frac{1}{2}\nabla \Psi(\gamma)\cdot S_{\gamma}(\varphi).
		\end{equation}
		From the local minimality of $\gamma$, it follows immediately that the right hand side in (\ref{positivoRHS}) is nonnegative. Finally, we can use the contractiveness of the projections $S_x$ to get the inequality
		\[\langle \gamma'', \varphi \rangle \le \int_{0}^{\delta}\frac{1}{2}|\nabla \Psi(\gamma)||\varphi|\quad \text{for every $\varphi \in C^{\infty}_{c}((0,\delta); \mathbb{R}^{d})$},\]
		and the thesis follows.
	\end{proof}
	\noindent As already shown above, in the comparison argument accompanying inequality (\ref{comparison}), from Lemma \ref{regvinc} descends the following
	\begin{cor}[Regularity in a potential zone]
		\label{regpotzone}
		Let $h$ be a nondecreasing potential shape, and let $\gamma$ be a local minimizer of $I^{\delta}_{f_{K},x_{0},x_{\delta}}$. Suppose that there exists an $\eta \in \mathcal{E}$ such that $\gamma((s,t))\subseteq Q_{\eta}$, where $(s,t)\subset [0,\delta]$. Then $\gamma$ is $C^{1,1}$ regular on $[s,t]$ and the following estimate holds
		\[|\ddot{\gamma}(r)|\le h'(|\gamma(r)-\eta|^2)|\gamma(r)-\eta|\quad \text{for a.e.\ $r\in (s,t)$}.\] 
	\end{cor} 
	\noindent Thus, point ii) of Theorem \ref{teoprinc} is proved. We can now address the second Lemma, which will turn out to be crucial in the proof of point i) of Theorem \ref{teoprinc}. 
	\begin{lem}[Reciprocal distance of intersecting polytopes]
		\label{distpol}
		Let $A,B \subset \mathbb{R}^{d}$ be two polytopes with $A\cap B \neq \emptyset$. Then there exists a sufficiently large constant $M>0$ such that
		\[\dist_{A\cap B}(x)\le M \dist_{B}(x)\quad \text{for every $x \in A$}.\] 
	\end{lem}
	\begin{proof}
		Let $P\subset \mathbb{R}^{d}$ be a polytope endowed with a representation of the form (\ref{rapprepoliedri}), where we assume without loss of generality that each of the affine functions $T_{j}$ has Lipschitz constant equal to $1$. 
		Let us consider the associated vector valued function $z_{P}:\mathbb{R}^{d}\rightarrow [0,+\infty)^{\ell}$ defined by
		\begin{equation}
			z_{P}(x)=\left(T_{1}(x)^{+},\dots,T_{\ell}(x)^{+}\right).
		\end{equation}
		As a first step we show that there exists a constant $c_P>0$ such that
		\begin{equation}\label{disugambrosio}
			T_{j}(x)^{+}\le \dist_{P}(x)\le c_{P}|z_{P}(x)|_{1}\quad \text{for every $j\in \left\{1,\dots,\ell \right\}$ and every $x\in \mathbb{R}^{d}$}.
		\end{equation}
		The left inequality follows by the $1$-Lipschitz assumption on $T_j$, while the right one can be obtained using a compactness argument as follows. Suppose by contradiction that there exists a sequence of points $x_n \in \mathbb{R}^{d}\setminus P$ such that 
		\[\dist_{P}(x_n)\ge n|z_{P}(x_n)|_{1}.\] 
		Up to replacing $x_n$ with 
		\[\pi_{P}(x_n)+ \frac{x_n - \pi_{P}(x_n)}{\dist_{P}(x_n)},\] 
		we can assume that $\dist_{P}(x_n)=1$. Now, if $x$ is any cluster point for $x_n$, we obtain that $\dist_{P}(x)=1$ and $z_{P}(x)=0$, which is clearly a contradiction. We are now in the position to prove the Lemma.
		Let $\ell \in \mathbb{N}$ be the number of affine functions in a representation of $B$ of the form (\ref{rapprepoliedri}), with $1$-Lipschitz affine functions $T_{j}$. By exploiting the estimates in (\ref{disugambrosio}), for every $x\in A$ we can bound 
		\[\dist_{A\cap B}(x)\le c_{A\cap B}|z_{A\cap B}(x)|_1\le \ell c_{A\cap B}\dist_{B}(x),\]
		where $c_{A\cap B}$ is a positive constant. The thesis follows by choosing $M=\ell c_{A\cap B}$. 
	\end{proof}
	
	\paragraph{Proof of the regularity result.}
	In this paragraph we are going to complete the proof of our main regularity result, Theorem \ref{teoprinc}, following the heuristic idea outlined above. It remains to prove point i), and we already restricted to show claim (\ref{claim}). Throughout the proof, $\gamma$ will be a fixed local minimizer of $I^{\delta}_{f_{K},x_{0},x_{\delta}}$ and $L$ the Lipschitz constant of $\gamma$. Moreover, for the sake of simplicity, we indicate $I^{\delta}_{f_{K},x_{0},x_{\delta}}$ as $\mathcal{F}$. As a preliminary observation, notice that we can choose an $R>0$ large enough such that $\gamma([0,\delta])\subset B_{R}$, and then, by point v) in Proposition \ref{vorcellsandpotzones}, we can find an $\alpha>0$ small enough such that
	\begin{equation}\label{prelimfact}
		|\bar{\eta}-x|\ge |\eta-x|+\alpha\quad \text{for every $x\in B_{2R}\cap Q_{\eta}\cap P_{\bar{\eta}}$ and every $\eta, \bar{\eta} \in \mathcal{E}$}.
	\end{equation} 
	\begin{proof}[Proof of claim \ref{claim}]
		First of all, by translation invariance, we can assume without loss of generality that $\gamma(t)=0$, thus simplifying some further computations. Then we consider the asymptotic lowest potential threshold 
		\[a:=\underset{s\rightarrow t^{-}}{\liminf}\,|\eta(\gamma(s))|.\]  
		We call $\tilde{\mathcal{E}}$ the subset of $\mathcal{E}$ indexing the potential zones that are visited infinitely often before $t$. Namely,
		\[\tilde{\mathcal{E}}:=\left\{\eta \in \mathcal{E}: \gamma^{-1}(Q_{\eta})\cap (0,t) \text{ accumulates in $t$}\right\}.\]
		Notice that the thesis is equivalent to $\# \tilde{\mathcal{E}}=1$. The set $\tilde{\mathcal{E}}$ can be partitioned into two subsets $\tilde{\mathcal{E}}_{1}$ and $\tilde{\mathcal{E}}_{2}$ defined by
		\begin{align*}
			\tilde{\mathcal{E}}_{1}&:=\left\{\eta\in \tilde{\mathcal{E}}: |\eta|=a\right\},\\
			\tilde{\mathcal{E}}_{2}&:=\left\{\eta\in \tilde{\mathcal{E}}: |\eta|>a\right\}.
		\end{align*} 
		The set $\tilde{\mathcal{E}}_{1}$, which is clearly nonempty by the definition of $a$, corresponds to those potential zones which are infinitely often visited by $\gamma$ before $t$ and that share the asymptotic lowest potential threshold $a$. We expect the curve $\gamma$ to spend most of the time there.
		
		We make the following technical choices of constants in order to simplify later arguments. We fix $r,\mu,\epsilon >0$ small enough so that these requirements are satisfied:
		\begin{itemize}
			\item[R1)] For every $x\in B_{r}$ we have $\opt(x)\subseteq \opt(0)$.
			\item[R2)] $\gamma((t-\epsilon,t))\subset B_{r}$.
			\item[R3)] For every $x\in B_{r}$, and for every $\eta \in \mathcal{E}$, we have $|\eta|-\mu <|x-\eta|<|\eta|+\mu$.
			\item[R4)] $3\mu \le \alpha$.
			\item[R5)] For every choice of $\eta, \bar{\eta} \in \mathcal{E}$, exactly one of the following holds:
			\[|\eta|=|\bar{\eta}|,\quad \min \left\{|\eta|, |\bar{\eta}|\right\} \le \max \left\{|\eta|, |\bar{\eta}|\right\}-3\mu.\]
			\item[R6)] $(t-\epsilon,t)\subseteq \underset{\eta \in \tilde{\mathcal{E}}}{\bigcup}\gamma^{-1}(Q_{\eta})$.
		\end{itemize}
		Notice that thanks to R1), for every $x\in B_{r}$ the segment $[x,0)$ will be entirely contained in the Voronoi cell $V_{\opt(x)}$, while R2) assures us that the curve belongs to this good area. Conditions R3), R4) and R5) will be useful to effectively distinguish between different potential zones. Finally, condition R6) implies that the potential zones touched by $\gamma$ in the interval $(t-\epsilon,t)$ are exactly those touched asymptotically. Thus, in particular, we have 
		\[a=\underset{s\in(t-\epsilon,t)}{\min}|\eta(\gamma(s))|.\]
		
		\noindent By condition R6), the interval $(t-\epsilon,t)$ can be partitioned into the following two sets:
		\begin{align*}
			C_{1}&:=(t-\epsilon,t)\cap \underset{\eta \in \tilde{\mathcal{E}}_1}{\bigcup}\gamma^{-1}(Q_{\eta}),\\
			C_{2}&:=(t-\epsilon,t)\cap \underset{\eta \in \tilde{\mathcal{E}}_2}{\bigcup}\gamma^{-1}(Q_{\eta}).
		\end{align*}
		We observe that $C_{1}$ is closed in $(t-\epsilon,t)$. In fact, if $s_j \in C_{1}$ and $s_j \rightarrow s_{\infty}\in (t-\epsilon,t)$, then by lower semicontinuity of the modulus of the extended gradient, we have $|\eta(\gamma(s_{\infty}))|\le a+\mu$, whence $|\eta(\gamma(s_{\infty}))|\le a+2\mu$, which implies $|\eta(\gamma(s_{\infty}))|= a$, that is $s_{\infty}\in C_1$. Then $C_{2}$ is open and can be written as a finite or countable (possibly empty) union of open intervals. 
		
		Remember that the thesis is equivalent to $\# \tilde{\mathcal{E}}=1$. We now assume that $\# \tilde{\mathcal{E}}\ge 2$ and try to find a contradiction.
		
		\noindent \textbf{\underline{Step 1}}. (Reduction to the case in which $\# \tilde{\mathcal{E}}_2 \ge 1$ and $\# \tilde{\mathcal{E}}_1 \ge 2$).
		
		\noindent We first show that $\# \tilde{\mathcal{E}}_2 \ge 1$. If this were not the case, then we would have $|\eta(\gamma(s))|=a$ for every $s\in (t-\epsilon,t)$ and $\# \tilde{\mathcal{E}}_1 \ge 2$. Then we could find two distinct $\eta, \bar{\eta}$ both belonging to $\tilde{\mathcal{E}}_{1}$ and a time $s\in (t-\epsilon, t)$ such that $x:=\gamma(s)\in Q_{\eta}\cap P_{\bar{\eta}}$. Now, by (\ref{prelimfact}), this would imply that 
		\[|x-\eta|\le |x-\bar{\eta}|-\alpha,\]
		and we reach a contradiction through the following chain of inequalities:
		\[a=|\eta|\le |x-\eta|+\mu \le |x-\bar{\eta}|-\alpha + \mu \le |\bar{\eta}|+2\mu -\alpha = a+2\mu - \alpha \le a-\mu<a.\]
		
		\noindent We now show that $\# \tilde{\mathcal{E}}_1 \ge 2$. Suppose by contradiction that $\bar{\mathcal{E}}_{1}=\left\{\eta\right\}$. Then we can build a better competitor $\tilde{\gamma}$ by performing arbitrarily small perturbations of $\gamma$ in the following way. We choose $s\in (t-\epsilon,t)$ as close to $t$ as we want, such that $\gamma(s)\in Q_{\eta}$. Then we modify $\gamma$ only in the interval $[s,t]$, by replacing it with its projection on the closed convex set $P_{\eta}$. Namely, denoting by $\pi_{\eta}$ the projection on $P_{\eta}$, we define
		\[
		\tilde{\gamma}(u)=\begin{cases}
			\gamma(u) &\text{for $u<s$ or $u>t$},\\
			\pi_{\eta}(\gamma(u))\quad &\text{for $u\in [s,t]$}.
		\end{cases}
		\]
		Now, we clearly have $|\dot{\tilde{\gamma}}|\le |\dot{\gamma}|$, by the $1$-Lipschitz property of $\pi_{\eta}$. On the other hand, for what concerns the potential, in the interval $[s,t]$ it holds
		\begin{align*}
			&|\nabla f(\gamma(u))|=|\nabla f(\tilde{\gamma}(u))| \quad \text{if $u\in C_{1}$},\\
			&|\nabla f(\gamma(u))|>|\nabla f(\tilde{\gamma}(u))| \quad \text{if $u\in C_{2}$}.
		\end{align*} 
		The second expression easily follows from the fact that if $s\in C_{2}$, then $|\eta(\gamma(s))|\ge a+3\mu$. From these estimates, the strict monotonicity of the potential shape $h$, and the local minimality of $\gamma$, we obtain that $C_{2}\cap (s,t)$ must be negligible for the one dimensional Lebesgue measure. But this contradicts the fact that $C_{2}\cap (s,t)$ is open and accumulates in $t$. Thus Step 1 is completed.
		
		To make the point, after Step 1 the situation is the following. There must exist
		\begin{itemize}
			\item Two elements $\eta,\bar{\eta} \in \bar{\mathcal{E}}_{1}$, with $\eta \neq \bar{\eta}$;
			\item Two distinct Voronoi cells $V_H$ and $V_{\bar{H}}$, with $H\subseteq Q_{\eta}$ and $\bar{H}\subseteq Q_{\bar{\eta}}$;
			\item A sequence of open intervals $(s_{\ell}, r_{\ell})\subset (t-\epsilon,t)$ accumulating in $t$ and such that, for every $\ell$, the following conditions hold:
			\begin{itemize}
				\item $r_{\ell}<s_{\ell+1}$;
				\item $\gamma(s_{\ell})\in V_{H}$; 
				\item $\gamma(r_{\ell})\in V_{\bar{H}}$;
				\item $(s_{\ell},r_{\ell})\subset C_{2}$.
			\end{itemize}
		\end{itemize} 
		\noindent \textbf{\underline{Step 2}}. (We have $(r_{\ell}-s_{\ell})=O((t-s_{\ell})^2)$).  
		
		\noindent This will be handled once again by comparison with a suitably chosen competitor. We define, for every $\ell$, the competitor $\gamma_{\ell}$, obtained by modifying $\gamma$ only in the interval $[s_{\ell},t]$, where $\gamma$ is replaced with its projection on the segment $[\gamma(s_{\ell}), 0]$. That is to say, calling $\pi_{\ell}$ the projection on $[\gamma(s_{\ell}), 0]$, we set
		\[
		\gamma_{\ell}(u)=
		\begin{cases}
			\gamma(u)  &\text{for $u<s_{\ell}$ or $u>t$},\\
			\pi_{\ell}(\gamma(u))\quad &\text{for $u\in [s_{\ell},t]$}.
		\end{cases}
		\]  
		Then $|\dot{\gamma}_{\ell}|\le |\dot{\gamma}|$ as before. Regarding the potential part, we have the following estimates, for $u\in (s_{\ell},t)$: 
		\begin{align*}
			&|\nabla f(\gamma(u))|^2 \ge |\eta(\gamma(u))|^2-2|\eta(\gamma(u))||\gamma(u)|-|\gamma(u)|^2\ge |\eta(\gamma(u))|^2-L(2S+r)(t-u),\\
			&|\nabla f(\gamma_{\ell}(u))|^2 \le |\eta(\gamma_{\ell}(u))|^2+2|\eta(\gamma_{\ell}(u))||\gamma_{\ell}(u)|+ |\gamma_{\ell}(u)|^2 \le |\eta(\gamma_{\ell}(u))|^2 +L(2S+r)(t-u),
		\end{align*} 
		where we set $S:=\underset{\eta \in \tilde{\mathcal{E}}}{\max}|\eta|$. We know that
		\begin{align*}
			&|\eta(\gamma(u))|^2 \ge a^2=|\eta(\gamma_{\ell}(u))|^2 \quad \text{for $u\in (s_{\ell},t)$},\\
			&|\eta(\gamma(u))|^2 \ge a^2+9\mu^2 \ge |\eta(\gamma_{\ell}(u))|^2 +9\mu^2 \quad \text{for $u\in (s_{\ell},r_{\ell})$}.
		\end{align*}
		\noindent Thus, after renaming the constant $c_0 := L(2S+r)$, we obtain, for sufficiently large $\ell$:
		\begin{align*}
			0 &\ge \mathcal{F}(\gamma)-\mathcal{F}(\gamma_{\ell})\\
			&\ge \int_{s_{\ell}}^{r_{\ell}}\left\{h(a^2 + 9\mu^{2}-c_{0}(t-u))-h(a^{2}+c_{0}(t-u))\right\}+\int_{r_{\ell}}^{t}\left\{h(a^2 -c_{0}(t-u))-h(a^{2}+c_{0}(t-u))\right\}\\
			&\ge c_{1}(r_{\ell}-s_{\ell})-c_{2}(t-s_{\ell})^{2}.
		\end{align*}
		Here $c_{1}:=h(a^{2}+8\mu^{2})-h(a^{2}+\mu^{2})>0$ and $\ell$ is large enough so that $c_{0}(t-s_{\ell})\le \mu^{2}$. The constant $c_{2}$ is instead defined as 
		\[c_{2}:=c_{0} \Lip\left(h; [a^{2}-\mu^{2}, a^{2}+\mu^{2}]\right)\] 
		By defining $c_{3}:=c_{2}/c_{1}$, we eventually get
		\[(r_{\ell}-s_{\ell})\le c_{3}(t-s_{\ell})^2\quad \text{for $\ell$ large enough}.\]
		
		Notice that $\overline{V_{H}}$ and $\overline{V_{\bar{H}}}$ are two polyhedra whose intersection contains $0$. Possibly replacing them with their intersection with a large $d$-dimensional cube, we can assume that they are polytopes. Then we can apply Lemma \ref{distpol} to deduce the existence of a constant $M>0$ and a sequence of points $x_{\ell}\in \overline{V_{H}}\cap \overline{V_{\bar{H}}}$ such that 
		\[|x_{\ell} - \gamma(s_{\ell})|\le M|\gamma(r_{\ell})-\gamma(s_{\ell})|\le MLc_{3}(t-s_{\ell})^{2}=c_{4}(t-s_{\ell})^2.\]
		\noindent \textbf{\underline{Step 3}}. (A slight deviation of $\gamma$ through $\overline{V_{H}}\cap \overline{V_{\bar{H}}}$ reduces the action, thus contradicting its local minimality).
		
		\noindent We call $\epsilon_{\ell} := c_{4}(t-s_{\ell})$ and assume that $\ell$ is large enough so that $\epsilon_{\ell}\in (0,1)$. We crucially consider the following competitor:
		\[
		\delta_{\ell}(u)=\begin{cases}
			\gamma(u) &\text{for $u<s_{\ell}$ or $u>t$},\\
			\gamma(s_{\ell})+ \frac{x_{\ell}-\gamma(s_{\ell})}{\epsilon_{\ell}(t-s_{\ell})}(u-s_{\ell}) &\text{for $u\in [s_{\ell},s_{\ell}+\epsilon_{\ell}(t-s_{\ell}))$},\\
			x_{\ell}+\frac{-x_{\ell}}{(1-\epsilon_{\ell})(t-s_{\ell})}(u-s_{\ell}-\epsilon_{\ell}(t-s_{\ell}))\quad &\text{for $u\in [s_{\ell}+\epsilon_{\ell}(t-s_{\ell}),t]$}.
		\end{cases}
		\] 
		Notice that in the interval $[s_{\ell}, t]$, the curve $\delta_{\ell}$ is simply a piecewise linear modification of $\gamma$, going from $\gamma(s_{\ell})$ to $x_{\ell}$ in time $\epsilon_{\ell}(t-s_{\ell})$, and then from $x_{\ell}$ to $0$ in the remaining time. We will see that $\delta_{\ell}$ has strictly less action than $\gamma$ for $\ell$ large enough, thus reaching the desired contradiction. We first want to be sure that $\overline{V_{H}}\cap \overline{V_{\bar{H}}}$ is a very low potential zone, so that we can lower the action of $\gamma$ by a slight deviation through it. This can be seen as follows. For sure $\eta$ and $\bar{\eta}$ both belong to $\partial f(x_{\ell})$, thus also $\frac{\eta +\bar{\eta}}{2} \in \partial f(x_{\ell})$. But then
		\[|\eta(x_{\ell})|\le |\eta(x_{\ell})-x_{\ell}|+|x_{\ell}|\le \left|\frac{\eta + \bar{\eta}}{2}-x_{\ell}\right|+|x_{\ell}|\le \left|\frac{\eta + \bar{\eta}}{2}\right|+2|x_{\ell}|=\left(a^2-\left|\frac{\eta-\bar{\eta}}{2}\right|^2\right)^{\frac{1}{2}}+2|x_{\ell}|,\]
		where we used in the very last equality that $|\eta|=|\bar{\eta}|=a$. Therefore, for $\ell$ large enough, we can assume that $x_{\ell} \in B_{r}$, and $|\eta(x_{\ell})|<a$. The last one in particular implies that
		\[|\eta(x)|\le a-3\mu \quad \text{for every $x\in [x_{\ell},0]$ and $\ell$ large enough}.\]
		We can also estimate
		\[|x_{\ell}|^2\le (|\gamma(s_{\ell})|+|x_{\ell}-\gamma(s_{\ell})|)^2\le |\gamma(s_{\ell})|^2+2L\epsilon_{\ell}(t-s_{\ell})^2+\epsilon_{\ell}^2(t-s_{\ell})^2.\] 
		Let us then compare the action of $\delta_{\ell}$ with the one of $\gamma$. We start from the kinetic part:
		\begin{align*}
			&|\dot{\delta}_{\ell}(u)|^2=\left(\frac{|x_{\ell}-\gamma(s_{\ell})|}{\epsilon_{\ell}(t-s_{\ell})}\right)^2\le 1\quad \text{for $u\in [s_{\ell},s_{\ell}+\epsilon_{\ell}(t-s_{\ell}))$},\\
			&|\dot{\delta}_{\ell}(u)|^2=\left(\frac{|x_{\ell}|}{(1-\epsilon_{\ell})(t-s_{\ell})}\right)^2\le \frac{|\gamma(s_{\ell})|^2}{(1-\epsilon_{\ell})^{2}(t-s_{\ell})^2}+\frac{2L\epsilon_{\ell}}{(1-\epsilon_{\ell})^2}+\frac{\epsilon_{\ell}^2}{(1-\epsilon_{\ell})^2}\quad \text{for $u\in [s_{\ell}+\epsilon_{\ell}(t-s_{\ell}),t)$}.
		\end{align*}
		Whence, integrating:
		\begin{align*}
			&\int_{s_{\ell}}^{t}|\dot{\gamma}|^2\ge \frac{|\gamma(s_{\ell})|^2}{t-s_{\ell}},\\
			&\int_{s_{\ell}}^{t}|\dot{\delta}_{\ell}|^2 \le \epsilon_{\ell}(t-s_{\ell})+\frac{|\gamma(s_{\ell})|^2}{(1-\epsilon_{\ell})(t-s_{\ell})}+\frac{2L\epsilon_{\ell}(t-s_{\ell})}{(1-\epsilon_{\ell})}+\frac{\epsilon_{\ell}^{2}(t-s_{\ell})}{(1-\epsilon_{\ell})}=\frac{|\gamma(s_{\ell})|^2}{(1-\epsilon_{\ell})(t-s_l)}+O(\epsilon_{\ell}^{2}).
		\end{align*}
		On the other hand, for the potential part we have
		\begin{align*}
			&|\nabla f(\gamma(u))|^2\ge a^2-c_{0}(t-u)\quad \text{for $u\in [s_{\ell},t)$},\\
			&|\nabla f(\delta_{\ell}(u))|^2\le a^{2}+c_{0}(t-u)\quad \text{for $u\in [s_{\ell},s_{\ell}+\epsilon_{\ell}(t-s_{\ell}))$},\\
			&|\nabla f(\delta_{\ell}(u))|^2\le a^{2}-9\mu^2 + c_{0}(t-u)\quad \text{for $u\in [s_{\ell}+\epsilon_{\ell}(t-s_{\ell}),t)$}.
		\end{align*}
		Integration yields
		\begin{align*}
			&\int_{s_{\ell}}^{t}h(|\nabla f(\gamma)|^2) \ge h(a^{2})(t-s_{\ell})-c_{5}(t-s_{\ell})^2,\\
			&\int_{s_{\ell}}^{t}h(|\nabla f(\delta_{\ell})|^2) \le h(a^{2})(t-s_{\ell})+c_{5}(t-s_{\ell})^2 - c_{6}(1-\epsilon_{\ell})(t-s_{\ell}).
		\end{align*}
		Here $\ell$ is choosen large enough so that $c_{0}(t-s_{\ell})\le \mu^2$. Moreover, we set 
		\[c_{5}:= \frac{c_{0}}{2}\Lip\left(h;[a^{2}-\mu,a^{2}+\mu]\right),\quad  c_{6}:=h(a^{2})-h(a^{2}-8\mu^{2})>0.\]
		Finally, collecting all the estimates together, we obtain
		\begin{align*}
			\mathcal{F}(\gamma)-\mathcal{F}(\delta_{\ell})&\ge \frac{|\gamma(s_{\ell})|^2}{t-s_{\ell}}\left(\frac{-\epsilon_{\ell}}{(1-\epsilon_{\ell})}\right) +c_{6}(1-\epsilon_{\ell})(t-s_{\ell})+O(\epsilon_{\ell}^{2})\\
			&\ge \Bigg(-\frac{L^2}{c_{4}(1-\epsilon_{\ell})}\epsilon_{\ell}+\frac{c_{6}}{c_{4}}(1-\epsilon_{\ell})+O(\epsilon_{\ell})\Bigg)\epsilon_{\ell}.
		\end{align*}
		Now the contradiction comes from the fact that the right term is strictly positive for $\ell$ large enough. This concludes the proof.
	\end{proof}

	\bibliographystyle{plain}

\begin{thebibliography}{10}
		
		\bibitem{AMBROSIO1989301}
		Luigi Ambrosio, Oscar Ascenzi, and Giuseppe Buttazzo.
		\newblock Lipschitz regularity for minimizers of integral functionals with
		highly discontinuous integrands.
		\newblock {\em Journal of Mathematical Analysis and Applications},
		142:301--316, 1989.
		
		\bibitem{AmbBrMAG}
		Luigi Ambrosio, Aymeric Baradat, and Yann Brenier.
		\newblock {M}onge-{A}mp\`ere gravitation as a {$\Gamma$}-limit of good rate
		functions.
		\newblock 2020.
		\newblock preprint, \url{ https://doi.org/10.48550/arXiv.2002.11966 }.
		
		\bibitem{AmbBrGAMMACONV}
		Luigi Ambrosio, Aymeric Baradat, and Yann Brenier.
		\newblock {$\Gamma$}-convergence for a class of action functionals induced by
		gradients of convex functions.
		\newblock {\em Rend. Lincei. Mat. Appl.}, 32:97--108, 2021.
		
		\bibitem{AmbBre21}
		Luigi Ambrosio and Camillo Brena.
		\newblock Stability of a class of action functionals depending on convex
		functions.
		\newblock {\em Disc. Cont. Dyn. Syst.}, 43:993--1005, 2023.
		
		\bibitem{AGS}
		Luigi Ambrosio, Nicola Gigli, and Giuseppe Savaré.
		\newblock {\em Gradient Flows in Metric Spaces and in the Space of Probability
			Measures}.
		\newblock Lectures in Mathematics ETH Z\"{u}rich. Birkh\"{a}user Verlag, Basel,
		2nd edition, 2008.
		
		\bibitem{Voronoi}
		Franz Aurenhammer.
		\newblock Voronoi diagrams, a survey of a fundamental geometric data structure.
		\newblock {\em ACM Comput. Surv.}, 23:345–405, 1991.
		
		\bibitem{Brenier1991PolarFA}
		Yann Brenier.
		\newblock Polar factorization and monotone rearrangement of vector-valued
		functions.
		\newblock {\em Communications on Pure and Applied Mathematics}, 44:375--417,
		1991.
		
		\bibitem{Brenier2015ADL}
		Yann Brenier.
		\newblock A double large deviation principle for {M}onge-{A}mp\'ere
		gravitation.
		\newblock {\em Bull. Inst. Math. Acad. Sin.}, 11:23--41, 2016.
		
		\bibitem{Clerc2020OnTV}
		Gauthier Clerc, Giovanni Conforti, and Ivan Gentil.
		\newblock On the variational interpretation of local logarithmic sobolev
		inequalities.
		\newblock 2020.
		\newblock preprint, \url{https://doi.org/10.48550/arXiv.2011.05207}.
		
		\bibitem{CorderoErausquin1999SurLT}
		Dario Cordero-Erausquin.
		\newblock Sur le transport de mesures p{\'e}riodiques.
		\newblock {\em C. R. Acad. Sci. Paris S\'{e}r. I Math}, 329:199--202, 1999.
		
		\bibitem{ConvAnal}
		Jean-Baptiste Hiriart-Urruty and Claude Lemaréchal.
		\newblock {\em Fundamentals of Convex Analysis}.
		\newblock Grundlehren Text Editions. Springer Berlin, Heidelberg, 1st edition,
		2001.
		
		\bibitem{McCann2001}
		Robert McCann.
		\newblock Polar factorization of maps on {R}iemannian manifolds.
		\newblock {\em Geom. Func. Anal.}, 11:589--608, 2001.
		
	\end{thebibliography}

\end{document}